\documentclass[10pt,reqno]{article}

\usepackage[b5paper,top=1.2in,left=0.9in]{geometry}
\usepackage{verbatim}
\usepackage{amssymb}
\usepackage{amsmath}
\usepackage{graphicx}
\usepackage{appendix}
\usepackage{color}
\usepackage{amsthm}
\usepackage{tikz}
\usetikzlibrary{arrows,positioning,decorations.pathmorphing,  decorations.markings}

\renewcommand{\d}{\mathrm{d}}
\newcommand{\D}{\mathrm{D}}
\newcommand{\e}{\mathrm{e}}

\newtheorem{Thm}{Theorem}[section]
\newtheorem{Lem}[Thm]{Lemma}
\newtheorem{Prop}[Thm]{Proposition}
\newtheorem{Cor}[Thm]{Corollary}
\newtheorem{Rem}[Thm]{Remark}
\newtheorem{Def}[Thm]{Definition}

\newtheorem{Ex}[Thm]{Example}

\newtheorem*{MainThm}{Main Theorem}

\newtheoremstyle{named}{}{}{\itshape}{}{\bfseries}{.}{.5em}{#1 #3}
\theoremstyle{named}

\def\R{\mathbb{R}}
\def\Q{\mathbb{Q}}

\def\C{\mathbb{C}}
\def\Z{\mathbb{Z}}

\def\fb{\mathfrak{b}}
\def\sl{\mathfrak{sl}}
\def\g{\mathfrak{g}}

\def\fh{\mathfrak{h}}

\def\cA{\mathcal{A}}
\def\cB{\mathcal{B}}

\def\cD{\mathcal{D}}

\def\cH{\mathcal{H}}

\def\cP{\mathcal{P}}

\def\cU{\mathcal{U}}

\def\a{\alpha}
\def\b{\beta}
\def\e{\epsilon}

\def\c{\gamma}
\def\D{\Delta}

\def\d{\delta}

\def\l{\lambda}

\def\s{\sigma}
\def\t{\tau}
\def\W{\Omega}

\def\ze{\zeta}

\def\bb{\textbf{b}}
\def\bc{\textbf{c}}

\def\be{\textbf{e}}

\def\bf{\textbf{f}}

\def\bH{\textbf{H}}
\def\bi{\textbf{i}}

\def\bQ{\textbf{Q}}

\def\bU{\textbf{U}}

\def\bW{\textbf{W}}
\def\zero{\textbf{0}}

\def\=>{\Longrightarrow}

\def\corr{\longleftrightarrow}
\def\iff{\Longleftrightarrow}

\def\to{\longrightarrow}

\def\ox{\otimes}
\def\o+{\oplus}
\def\bo+{\bigoplus}

\def\<{\langle}
\def\>{\rangle}
\def\({\left(}
\def\){\right)}
\def\oo{\infty}

\def\^{\wedge}
\def\+{\dagger}

\def\inv{^{-1}}
\def\half{\frac{1}{2}}

\def\dd[#1,#2]{\frac{d#1}{d#2}}
\def\del[#1,#2]{\frac{\partial #1}{\partial #2}}
\def\over[#1]{\overline{#1}}
\def\vec[#1]{\overrightarrow{#1}}

\def\tab{\;\;\;\;\;\;}

\newcommand{\til}[1]{\widetilde{#1}}
\newcommand{\what}[1]{\widehat{#1}}

\newcommand{\veca}[2][cccccccccccccccccccccccccccccccccccccccccc]{\left(\begin{array}{#1}#2 \\ \end{array} \right)}

\newcommand{\case}[2][cccccccccccccccccccccccccccccccccccccccccc]{\left\{\begin{array}{#1}#2 \\ \end{array}\right.}
\newcommand{\Eq}[1]{\begin{align}#1\end{align}}
\newcommand{\Eqn}[1]{\begin{align*}#1\end{align*}}

\begin{document}
\title{Positive Representations of Split Real Quantum Groups: The Universal $R$ Operator}

\author{  Ivan C.H. Ip\footnote{
         	Kavli Institute for the Physics and Mathematics of the Universe (WPI), 
		The University of Tokyo, 
		Kashiwa, Chiba 
		277-8583, Japan
		\newline
		Email: ivan.ip@ipmu.jp
          }
}

\date{\today}

\numberwithin{equation}{section}

\maketitle

\begin{abstract}
The universal $R$ operator for the positive representations of split real quantum groups is computed, generalizing the formula of compact quantum groups $\cU_q(\g)$ by Kirillov-Reshetikhin and Levendorski\u{\i}-Soibelman, and the formula in the case of $\cU_{q\til{q}}(\sl(2,\R))$ by Faddeev, Kashaev and Bytsko-Teschner. Several new functional relations of the quantum dilogarithm are obtained, generalizing the quantum exponential relations and the pentagon relations. The quantum Weyl element and Lusztig's isomorphism in the positive setting are also studied in detail.  Finally we introduce a $C^*$-algebraic version of the split real quantum group in the language of multiplier Hopf algebras, and consequently the definition of $R$ is made rigorous as the canonical element of the Drinfeld's double $\bU$ of certain multiplier Hopf algebra $\bU\bb$. Moreover a ribbon structure is introduced for an extension of $\bU$.
\end{abstract}


\newpage
\tableofcontents

\section{Introduction}\label{sec:intro}

In this paper, we construct the universal $R$ operator for the positive representations of split real quantum groups $\cU_{q\til{q}}(\g_\R)$, generalizing the formula of the $R$-operator in the case of $\cU_{q\til{q}}(\sl(2,\R))$ by Faddeev \cite{Fa2}, Kashaev \cite{Ka} and Bytsko-Teschner \cite{BT}, as well as the universal $R$-matrix computed independently by Kirillov-Reshetikhin \cite{KR} and Levendorski\u{\i}-Soibelman \cite{LS} for compact quantum group $\cU_q(\g)$ associated to some Lie algebra $\g$ of all type. 

The notion of the \emph{positive principal series representations}, or simply \emph{positive representations}, was introduced in \cite{FI} as a new research program devoted to the representation theory of split real quantum groups. It uses the concept of modular double for quantum groups \cite{Fa1, Fa2}, and has been studied for $U_{q\til{q}}(\sl(2,\R))$ by Teschner \textit{et al.} \cite{BT, PT1, PT2}. Explicit construction of the positive representations $\cP_\l$ has been obtained for the simply-laced case in \cite{Ip2} and non-simply-laced case in \cite{Ip3}, where the generators of the quantum groups are realized by positive essentially self-adjoint operators. Furthermore, the so-called \emph{transcendental relations} of the (rescaled) generators:
\Eq{\til{\be_i}=\be_i^{\frac{1}{b_i^2}},\tab \til{\bf_i}=\bf_i^{\frac{1}{b_i^2}}, \tab \til{K_i}=K_i^{\frac{1}{b_i^2}}}
 gives the self-duality between different parts of the modular double, while in the non-simply-laced case, new explicit analytic relations between the quantum group and its Langland's dual have been observed \cite{Ip3}.

Motivated by the detailed study in the case of $\cU_{q\til{q}}(\sl(2,\R))$ by Teschner \textit{et al.}, a natural problem is to find the universal $R$-matrix so that it gives a braiding of the positive representations $\cP_{\l}$ of the split real quantum groups $\cU_{q\til{q}}(\g_\R)$. In this infinite dimensional setting, we do not expect any ``matrix" anymore, but rather a natural setting will be realizing $R$ as a \emph{unitary operator} acting on $\cP_{\l_1}\ox \cP_{\l_2}$ such that the usual properties are satisfied:
\begin{itemize}
\item The braiding relation
\Eq{\D'(X)R:=(\s \circ \D)(X)R=R\D(X), \tab \s(x\ox y)=y\ox x\label{br};}
\item The quasi-triangularity
\Eq{(\D \ox id)(R)=&R_{13}R_{23}\label{qt1},\\
(id\ox\D)(R)=&R_{13}R_{12}\label{qt2}.
}\end{itemize}
Here the coproduct $\D$ acts on $R$ in a natural way on the generators, and we have also used the standard leg notation. These together imply the Yang-Baxter equation
\Eq{R_{12}R_{13}R_{23}=R_{23}R_{13}R_{12}.\label{YB}}

The expression of $R$ in the case of $\cU_{q\til{q}}(\sl(2,\R))$ is particularly simple, and is given by
\Eq{R=q^{\frac{H\ox H}{4}} g_b(\be\ox \bf)q^{\frac{H\ox H}{4}},}
where 
\Eq{E=\frac{\bi}{q-q\inv}\be, \tab F=\frac{\bi}{q-q\inv}\bf,\tab K=q^H \label{bigEF}} are the usual generators, and $g_b(x)$ is the remarkable quantum dilogarithm function, central to the study of split real quantum groups. See also \cite{BT} for a discussion why this operator deserves to be called ``universal", although we won't consider this aspect in the present paper.

On the other hand, the universal $R$-matrix in the compact case is given by products of the form
\Eq{\bQ^\half\prod_\a Exp_{q^{-2}}((1-q^{-2})E_\a\ox F_\a)\bQ^\half, \label{Exp1}}
where $\bQ=q^{\sum (dA\inv)_{ij}H_i\ox H_j}$ with $dA$ the symmetrized Cartan matrix. Here $Exp_q(x)$ is the quantum exponential function, and $E_\a$ are the root vectors for $\g$, given by the Lusztig's isomorphism $T_k$ on the simple root vectors, which can be written as certain composition of $q$-commutators, and play a crucial role in the theory of Lusztig's canonical basis \cite{Lu2}.

Therefore a natural proposal will be replacing the expression \eqref{Exp1} by
\Eq{\bQ^\half \prod_\a g_b(\be_\a\ox \bf_\a)\bQ^\half,} 
thus generalizing both equations. More precisely, we have

\begin{MainThm} Let $w_0=s_{i_1}s_{i_2}...s_{i_N}$ be a reduced expression of the longest element of the Weyl group. Then the universal $R$-operator for the positive representations of $\cU_{q\til{q}}(\g_\R)$ acting on $\cP_{\l_1}\ox \cP_{\l_2}\simeq L^2(\R^N)\ox L^2(\R^N)$ is a unitary operator given by
\Eq{\label{introR}R=&\prod_{ij}q_i^{\half(A\inv)_{ij}H_i\ox H_j}\prod_{k=1}^{N} g_b(\be_{\a_k}\ox \bf_{\a_k}) \prod_{ij}q_i^{\half(A\inv)_{ij}H_i\ox H_j},}
where $\be_{\a_k}=T_{i_1} T_{i_2}... T_{i_{k-1}} \be_{i_k}$ is given by the Lusztig's isomorphism in Theorem \ref{LusThm}, similarly for $\bf_{\a_k}$. The product is such that the term $k=1$ appears on the rightmost side.
\end{MainThm}
In particular, by the properties of the transcendental relations \cite{Ip2,Ip3} as well as the self-duality of $g_b(x)$, the universal $R$ operator simultaneously serves as an $R$ operator for the modular double counterpart.

The main difficulty lies in the fact that, in order for the expression \eqref{introR} to be well-defined, we need both $\be_\a$ and $\bf_\a$ to be positive essentially self-adjoint, so that we can apply functional calculus. Following the approach by \cite{KR} and \cite{LS}, the main technical result is that these non-simple basis can actually be obtained by conjugations on the generators by means of the quantum Weyl elements $w_i$, which is unitary in the setting of positive representations (cf. Corollary \ref{Tpos}).
\begin{Thm} The operators $\be_{\a_k}$ and $\bf_{\a_k}$ corresponding to non-simple roots are positive essentially self-adjoint under the positive representations, and satisfy the transcendental relations.
\end{Thm}
Because of the nice properties enjoyed by the \emph{rescaled} generators $\be_i$ and $\bf_i$, we find it instructive throughout the paper to stick with these variables rather than the original $E_i$ and $F_i$ as defined in \eqref{bigEF}.

Another difficulty lies in the fact that since the representations are infinite dimensional, and $|q|=1$, we can no longer work with formal power series and the quantum exponential function. In particular, the usual Drinfeld's double construction trick does not really work anymore. Instead, using hard technical analysis, we discover explicitly certain (considerably new) functional relations (cf. Proposition \ref{genpenta}-\ref{genexp}) of the quantum dilogarithm function $g_b(x)$, and prove directly the braiding relations and the quasi-triangular relations of the $R$ operator.

In order to compute the quantum Weyl elements, we have to compute the branching rules for $\cU_{q\til{q}}(\sl(2,\R))\subset \cU_{q\til{q}}(\g_\R)$. It turns out that the branching rules is particularly simple, and remarkably it resembles both the decomposition of the tensor product representation $P_\a\ox P_\b$ (cf. \cite{PT2}) and the Peter-Weyl type decomposition of $L^2(SL_q^+(2,\R))$ (cf. \cite{Ip1}) with exactly the same Plancherel measure. (cf. Theorem \ref{branchingrules}).

\begin{Thm}
Fixed any positive representation $\cP_\l\simeq L^2(\R^N)$ of $\cU_{q\til{q}}(\g_\R)$. As a representation of $\cU_{q\til{q}}(\sl(2,\R))\subset \cU_{q\til{q}}(\g_\R)$ corresponding to the root $\a_i$, 
\Eq{\cP_{\l}\simeq L^2(\R^{N-2})\ox \int_{\R_+} P_\c d\mu(\c)}
is a unitary equivalence, where $P_\c$ is the positive representation of $\cU_{q\til{q}}(\sl(2,\R))$ with parameter $\c$, and $d\mu(\c)=|S_{b_i}(Q_i+2\c)|^2d\c$.
\end{Thm}

Furthermore, we also encounter the calculation of the ribbon element $v$, and the element $u$ which exists for any (regular) quasi-triangular Hopf algebra. Therefore it is strongly suggested that there is an underlying algebraic structure enveloping all the calculations so far. In particular, the expression for the universal $R$-operator suggest that it is a canonical element of certain algebra with a ``continuous basis", very similar to the analysis that has been done for the quantum plane in our previous work \cite{Ip1}. Therefore we proceed to construct the split real quantum group in the $C^*$-algebraic setting, and show that in fact the satisfactory answer lies in the language of a \emph{multiplier Hopf algebra}, introduced by van Daele \cite{VD}. Consequently, all the calculations made so far are rigorously defined and simplified by the following (cf. Corollary \ref{mainRCor}):
\begin{Thm} The universal $R$ operator from the Main Theorem can be considered as (the projection of) the canonical element of the Drinfeld's double (cf. \cite{DvD, DvD2}) $\cD(\bU\bb)$ of the multiplier Hopf algebraic version of the Borel subalgebra $\bU\bb$,
\end{Thm}

Finally, we remark that the ribbon element $v$ calculated are also of certain interest, since the expression involves the number $Q=b+b\inv$, which implies that there is no classical limit as $b\to 0$. Hence this ribbon element differs from the one usually considered in compact quantum group, and it is well-known that the ribbon structure of Hopf algebra is needed to construct quantum topological invariant by the Reshetikhin-Turaev construction \cite{RT1, RT2}. Therefore this may serve as evidence for the possibility of constructing new classes of topological invariants.

The paper is organized as follows. Section \ref{sec:prelim} serves as the technical backbone of the paper. We fix the notations by recalling the definition of $\cU_q(\g)$ associated to simple Lie algebra $\g$ of general type. Next we recall the main properties and construction of the positive representations considered in \cite{FI, Ip2, Ip3}, and write down explicitly a particular expression for the rank=2 case. Since the paper involves a lot of technical computations, we review in detail the definition and properties of the quantum dilogarithm function $G_b$ and its variant $g_b$, which summarizes old and new results from \cite{BT, Ip1, Ip3} that is needed in this paper. Finally we recall the construction of the universal $R$-matrices by \cite{KR} and \cite{LS} in the compact quantum group case, as well as the universal $R$ operator by \cite{BT} in the case of $\cU_{q\til{q}}(\sl(2,\R))$.

In Section \ref{sec:qdlog}, we extend the quantum exponential relations and pentagon relations of $g_b(x)$ to more generalized setting involving certain $q$-commutators. These new functional relations are what we needed to prove the properties of the $R$ operator. In Section \ref{sec:QWeyl}, we proceed to construct the quantum Weyl elements so that conjugations by them realize Lusztig's isomorphism.  It involves calculating the ribbon element, and the branching rules of $\cU_{q\til{q}}(\sl(2,\R))\subset \cU_{q\til{q}}(\g_\R)$. In Section \ref{sec:R}, we state the main theorem about the universal $R$ operator, and prove the braiding relations and quasi-triangularity in the simply-laced case, while we only give several remarks on the non-simply-laced case to avoid getting too technical. Finally in Section \ref{sec:Cstar}, we introduce the notion of a multiplier Hopf algebra, and by finding certain Hopf pairing, we show that the universal $R$ operator can actually be regarded as the canonical element of a Drinfeld's double construction of the Borel subalgebra as a multiplier Hopf algebra, and we introduce a ribbon structure in the extension of the split real quantum group.

\textbf{Acknowledgments.} I would like to thank my advisor Prof I. Frenkel for ongoing support of this research program. I would also like to thank Changzheng Li for helpful discussion. This work was supported by World Premier International Research Center Initiative (WPI Initiative), MEXT, Japan.

\section{Preliminaries}\label{sec:prelim}
Throughout the paper, we will fix once and for all $q=e^{\pi \bi b^2}$ with $\bi =\sqrt{-1}$, \\$0<b^2<1$ and $b\in\R\setminus\Q$. We also denote by $Q=b+b\inv$.
\subsection{Definition of $\cU_q(\g)$}\label{sec:prelim:Uqgr}
In order to fix the convention we use throughout the paper, we recall the definition of the quantum group $\cU_q(\g_\R)$ where $\g$ is of general type \cite{CP}.
Let $I=\{1,2,...,n\}$ denote the set of nodes of the Dynkin diagram of $\g$ where $n=rank(\g)$.
\begin{Def} \label{qi} Let $(-,-)$ be the inner product of the root lattice. Let $\a_i$, $i\in I$ be the positive simple roots, and we define
\begin{eqnarray}
a_{ij}=\frac{2(\a_i,\a_j)}{(\a_i,\a_i)},\\
q_i:=q^{\frac{1}{2}(\a_i,\a_i)}:=e^{\pi \bi  b_i^2},
\end{eqnarray}
where $A=(a_{ij})$ is the Cartan matrix. We will let $\a_1$ be the short root in type $B_n$ and the long root in type $C_n, F_4$ and $G_2$.

We choose 
\Eq{\frac{1}{2}(\a_i,\a_i)=\case{1&\mbox{$i$ is long root or in the simply-laced case,}\\\frac{1}{2}&\mbox{$i$ is short root in type $B,C,F$,}\\\frac{1}{3}&\mbox{$i$ is short root in type $G_2$,}}}and $(\a_i,\a_j)=-1$ when $i,j$ are adjacent in the Dynkin diagram.

Therefore in the case when $\g$ is of type $B_n$, $C_n$ and $F_4$, if we define $b_l=b$, and $b_s=\frac{b}{\sqrt{2}}$ we have the following normalization:
\begin{eqnarray}
\label{qiBCF}q_i=\case{e^{\pi \bi b_l^2}=q&\mbox{$i$ is long root,}\\e^{\pi \bi b_s^2} =q^{\frac{1}{2}}&\mbox{$i$ is short root.}}
\end{eqnarray}
In the case when $\g$ is of type $G_2$, we define $b_l=b$, and $b_s=\frac{b}{\sqrt{3}}$, and we have the following normalization:
\begin{eqnarray}
\label{qiG}q_i=\case{e^{\pi \bi b_l^2}=q&\mbox{$i$ is long root,}\\e^{\pi \bi b_s^2} =q^{\frac{1}{3}}&\mbox{$i$ is short root.}}
\end{eqnarray}
\end{Def}

\begin{Def} Let $A=(a_{ij})$ denote the Cartan matrix. Then $\cU_q(\g)$ with $q=e^{\pi\bi b_l^2}$ is the algebra generated by $E_i$, $F_i$ and $K_i^{\pm1}$, $i\in I$ subject to the following relations:
\begin{eqnarray}
K_iE_j&=&q_i^{a_{ij}}E_jK_i,\\
K_iF_j&=&q_i^{-a_{ij}}F_jK_i,\\
{[E_i,F_j]} &=& \d_{ij}\frac{K_i-K_i\inv}{q_i-q_i\inv},
\end{eqnarray}
together with the Serre relations for $i\neq j$:
\begin{eqnarray}
\sum_{k=0}^{1-a_{ij}}(-1)^k\frac{[1-a_{ij}]_{q_i}!}{[1-a_{ij}-k]_{q_i}![k]_{q_i}!}E_i^{k}E_jE_i^{1-a_{ij}-k}&=&0,\label{SerreE}\\
\sum_{k=0}^{1-a_{ij}}(-1)^k\frac{[1-a_{ij}]_{q_i}!}{[1-a_{ij}-k]_{q_i}![k]_{q_i}!}F_i^{k}F_jF_i^{1-a_{ij}-k}&=&0,\label{SerreF}
\end{eqnarray}
where $[k]_q=\frac{q^k-q^{-k}}{q-q\inv}$. We also define $H_i$ so that $K_i=q_i^{H_i}$.
\end{Def}

The Hopf algebra structure of $\cU_q(\g)$ is given by
\Eq{
\D(E_i)=&K_i^{-\half}\ox E_i+E_i\ox K_i^{\half},\\
\D(F_i)=&K_i^{-\half}\ox F_i+F_i\ox K_i^{\half},\\
\D(K_i)=&K_i\ox K_i,\\
\e(E_i)=&\e(F_i)=0,\tab \e(K_i)=1,\\
S(E_i)=&-q_iE_i, \tab S(F_i)=-q_i\inv F_i,\tab S(K_i)=K_i\inv.
}

We define $\cU_q(\g_\R)$ to be the real form of $\cU_q(\g)$ induced by the star structure
\Eq{E_i^*=E_i,\tab F_i^*=F_i,\tab K_i^*=K_i.}
Finally, according to the results of \cite{Ip2,Ip3}, we define the modular double $\cU_{q\til{q}}(\g_\R)$ to be
\Eq{\cU_{q\til{q}}(\g_\R):=\cU_q(\g_\R)\ox \cU_{\til{q}}(\g_\R)&\tab \mbox{$\g$ is simply-laced,}\\
\cU_{q\til{q}}(\g_\R):=\cU_q(\g_\R)\ox \cU_{\til{q}}({}^L\g_\R)&\tab \mbox{otherwise,}}
where $\til{q}=e^{\pi \bi b_s^{-2}}$, and ${}^L\g_\R$ is the Langland's dual of $\g_\R$ obtained by interchanging the long roots and short roots of $\g_\R$.

\subsection{Positive representations of $\cU_{q\til{q}}(\g_\R)$}\label{sec:prelim:pos}
In \cite{FI, Ip2,Ip3}, a special class of representation for $\cU_{q\til{q}}(\g_\R)$, called the positive representation, is defined. The generators of the quantum groups are realized by positive essentially self-adjoint operators, and also satisfy the so-called \emph{transcendental relations}, relating the quantum group with its modular double counterpart. More precisely, we have
\begin{Thm} Let 
\Eq{\be_i:=2\sin(\pi b_i^2)E_i,\tab \bf_i:=2\sin(\pi b_i^2)F_i.\label{smallef}}
Note that $2\sin(\pi b_i^2)=\left(\frac{\bi}{q_i-q_i\inv}\right)\inv$.
Then there exists a representation $P_{\l}$ of $\cU_{q\til{q}}(\g_\R)$ parametrized by the $\R_+$-span of the cone of positive weights $\l\in P_\R^+$, or equivalently by $\l\in \R_+^n$ where $n=rank(\g)$, such that 
\begin{itemize}
\item The generators $\be_i,\bf_i,K_i$ are represented by positive essentially self-adjoint operators acting on $L^2(\R^{l(w_0)})$, where $l(w_0)$ is the length of the longest element $w_0\in W$ of the Weyl group.
\item Define the transcendental generators:
\Eq{\til{\be_i}:=\be_i^{\frac{1}{b_i^2}},\tab \til{\bf_i}:=\bf_i^{\frac{1}{b_i^2}},\tab \til{K_i}:=K_i^{\frac{1}{b_i^2}}.\label{transdef}}
Then \begin{itemize}
\item if $\g$ is simply-laced, the generators $\til{\be_i},\til{\bf_i},\til{K_i}$ is obtained by replacing $b$ with $b\inv$ in the representation of the generators $\be_i,\bf_i,K_i$. \\
\item If $\g$ is of type $B,C,F,G$, then the generators $\til{E_i},\til{F_i},\til{K_i}$ with
\Eq{\til{\be_i}:=2\sin(\pi b_i^{-2})\til{E_i},\tab \til{\bf_i}:=2\sin(\pi b_i^{-2})\til{F_i}}
generates $\cU_{\til{q}}({}^L\g_\R)$ defined in the previous section.
\end{itemize}
\item The generators $\be_i,\bf_i,K_i$ and $\til{\be_i},\til{\bf_i},\til{K_i}$ commute weakly up to a sign.
\end{itemize}
\end{Thm}

The positive representations are constructed for each reduced expression $w_0\in W$ of the longest element of the Weyl group, and representations corresponding to different reduced expression are unitary equivalent.
\begin{Def} Fixed a reduced expression of $w_0=s_{i_1}...s_{i_N}$. Let the coordinates of $L^2(\R^N)$ be denoted by $\{u_i^k\}$ so that $i$ is the corresponding root index, and $k$ denotes the sequence this root is appearing in $w_0$ from the right. Also denote by $\{v_j\}_{j=1}^N$ the same set of coordinates counting from the left, $v(i,k)$ the index such that $u_i^k=v_{v(i,k)}$, and $r(k)$ the root index corresponding to $v_k$.
\end{Def}
\begin{Ex} The coordinates of $L^2(\R^6)$ for $A_3$ corresponding to $w_0=s_3s_2s_1s_3s_2s_3$ is given by
$$(u_3^3,u_2^2,u_1^1,u_3^2, u_2^1,u_3^1)=(v_1,v_2,v_3,v_4,v_5,v_6).$$
\end{Ex} 
\begin{Def}\label{usul}Denote by 
\Eq{[u_s+u_l]e(-p_s-p_l):=e^{\pi b_s(-u_s-2p_s)+\pi b_l(-u_l-2p_l)}+e^{\pi b_s(u_s-2p_s)+\pi b_l(u_l-2p_l)},}
where $u_s$ is a linear combination of the variables corresponding to short roots, while $u_l$ is a linear combination of the variables corresponding to long roots, or variables in the simply-laced case. (The parameters $\l_i$ are also considered in both cases.)
\end{Def}
\begin{Thm}\cite{Ip2,Ip3} For a fixed reduced expression of $w_0$, the positive representation is given by
\Eq{
\bf_i=&\sum_{k=1}^n\left[-\sum_{j=1}^{v(i,k)-1} a_{i,r(j)}v_j-u_i^k-2\l_i\right]e(p_{i}^{k}),\label{FF}\\
K_i=&e^{-\pi (\sum_{k=1}^{l(w_0)} a_{i,r(k)}b_{r(k)}v_k+2b_\bi \l_i)},\label{KK}
}
and by taking $w_0=w's_i$ so that the simple reflection for root $i$ appears on the right, the action of $\be_i$ is given by
\Eq{\label{EE} 
\be_i=&[u_i^1]e(-p_i^1).}
\end{Thm}

In this paper, it is instructive to recall the explicit expression in the case of rank 1 and 2. For details of the construction and the other cases please refer to \cite{Ip2,Ip3}. 

\begin{Prop}\cite{BT,PT2}\label{canonicalsl2} The positive representation $P_\l$ of $\cU_{q\til{q}}(\sl(2,\R))$ is given by
\Eqn{
\be=&[u-\l]e(-p)=e^{\pi b(-u+\l-2p)}+e^{\pi b(u-\l-2p)},\\
\bf=&[-u-\l]e(p)=e^{\pi b(u+\l+2p)}+e^{\pi b(-u-\l+2p)},\\
K=&e^{-2\pi bu}.
}
(Note that it is related to the canonical form \eqref{FF}-\eqref{EE} by $u\mapsto u+\l$.)
\end{Prop}
\begin{Prop}\cite{Ip2}\label{typeA2} The positive representation $\cP_\l$ of $\cU_{q\til{q}}(\sl(3,\R))$ with parameters $\l=(\l_1,\l_2)$, corresponding to the reduced expression $w_0=s_2s_1s_2$, acting on $f(u,v,w)\in L^2(\R^3)$, is given by
\Eqn{
\be_1=&[v-w]e(-p_v)+[u]e(-p_v+p_w-p_u),\\
\be_2=&[w]e(-p_w),\\
\bf_1=&[-v+u-2\l_1]e(p_v),\\
\bf_2=&[-2u+v-w-2\l_2]e(p_w)+[-u-2\l_2]e(p_u),\\
K_1=&e^{-\pi b(-u+2v-w+2\l_1)},\\
K_2=&e^{-\pi b(2u-v+2w+2\l_2)}.
}
\end{Prop}

\begin{Prop}\cite{Ip3} The positive representation  $\cP_\l$ of $\cU_{q\til{q}}(\g_\R)$ with parameters $\l=(\l_1,\l_2)$, where $\g_\R$ is of type $B_2$, corresponding to the reduced expression $w_0=s_1s_2s_1s_2$, acting on $f(t,u,v,w)\in L^2(\R^4)$, is given by
\Eqn{
\be_1=&[t]e(-p_t-p_u+p_w)+[u-v]e(-p_u-p_v+p_w)+[v-w]e(-p_v),\\
\be_2=&[w]e(-p_w),\\
\bf_1=&[2\l_1-t]e(p_t)+[2\l_1-2t+u-v]e(p_v),\\
\bf_2=&[2\l_2+2t-u]e(p_u)+[2\l_2+2t-2u+2v-w]e(p_w),\\
K_1=&e^{\pi b_s(2\l_1-2t-2v)}e^{\pi b(u+w)},\\
K_2=&e^{\pi b(2\l_2-2u-2w)}e^{\pi b_s(2t+2v)}.
}
In this case (cf. Definition \ref{usul}), $u_s$ are linear combinations of $\{t,v\}$, while $u_l$ are linear combinations of $\{u,w\}$. Similarly for $p_s$ and $p_l$. 
\end{Prop}

We will omit the case of type $G_2$ for simplicity.

\subsection{Quantum dilogarithm $G_b(x)$ and $g_b(x)$}\label{sec:prelim:qdlog}
First introduced by Faddeev \cite{Fa1, Fa2}, the quantum dilogarithm $G_b(x)$ and its variants $g_b(x)$ play a crucial role in the study of positive representations of split real quantum groups, and also appear in many other areas of mathematics and physics. In this subsection, let us recall the definition and some properties of the quantum dilogarithm functions \cite{BT, Ip1, PT2} that is needed in the calculations in this paper.

\begin{Def} The quantum dilogarithm function $G_b(x)$ is defined on\\ ${0\leq Re(z)\leq Q}$ by
\Eq{\label{intform} G_b(x)=\over[\ze_b]\exp\left(-\int_{\W}\frac{e^{\pi tz}}{(e^{\pi bt}-1)(e^{\pi b\inv t}-1)}\frac{dt}{t}\right),}
where \Eq{\ze_b=e^{\frac{\pi \bi }{2}(\frac{b^2+b^{-2}}{6}+\frac{1}{2})},}
and the contour goes along $\R$ with a small semicircle going above the pole at $t=0$. This can be extended meromorphically to the whole complex plane with poles at $x=-nb-mb\inv$ and zeros at $x=Q+nb+mb\inv$, for $n,m\in\Z_{\geq0}$;
\end{Def}

The quantum dilogarithm $G_b(x)$ satisfies the following properties:
\begin{Prop} Self-Duality:
\Eq{G_b(x)=G_{b\inv}(x);\label{selfdual}}

Functional equations: \Eq{\label{funceq}G_b(x+b^{\pm 1})=(1-e^{2\pi \bi b^{\pm 1}x})G_b(x);}

Reflection property:
\Eq{\label{reflection}G_b(x)G_b(Q-x)=e^{\pi \bi x(x-Q)};}

Complex Conjugation: \Eq{\overline{G_b(x)}=\frac{1}{G_b(Q-\bar{x})},\label{Gbcomplex}}
in particular \Eq{\label{gb1}\left|G_b(\frac{Q}{2}+\bi x)\right|=1 \mbox{ for $x\in\R$};}

\label{asymp} Asymptotic Properties:
\Eq{G_b(x)\sim\left\{\begin{array}{cc}\bar{\ze_b}&Im(x)\to+\oo\\\ze_b
e^{\pi \bi x(x-Q)}&Im(x)\to-\oo\end{array},\right.}
\end{Prop}

\begin{Lem}[$q$-binomial theorem]\label{qbinom}
For positive self-adjoint variables $U,V$ with $UV=q^2VU$, we have:
\Eq{(U+V)^{\bi b\inv t}=\int_{C}\veca{\bi t\\\bi \t}_b U^{\bi b\inv (t-\t)}V^{\bi b\inv \t}d\t ,}
where the $q$-beta function (or $q$-binomial coefficient) is given by
\Eq{\veca{t\\\t}_b=\frac{G_b(-\t)G_b(\t-t)}{G_b(-t)},}
and $C$ is the contour along $\R$ that goes above the pole at $\t=0$
and below the pole at $\t=t$.
\end{Lem}

\begin{Lem}[tau-beta theorem] \label{tau}We have
\Eq{\int_C e^{-2\pi \t \b}
\frac{G_b(\a+\bi \t)}{G_b(Q+\bi \t)}d\t =\frac{G_b(\a)G_b(\b)}{G_b(\a+\b)},}
where the contour $C$ goes along $\R$ and goes above the poles of
$G_b(Q+\bi \t)$ and below those of $G_b(\a+\bi \t)$. By the asymptotic properties of $G_b$, the integral converges for $Re(\b)>0, Re(\a+\b)<Q$.
\end{Lem}
Generalizing the delta distribution results from \cite[Cor 3.13]{Ip1}, we have the following
\begin{Prop}\label{delta} For $f(x)$ entirely analytic and rapidly decreasing (faster than any exponential) along the real direction, we have
\Eq{&\lim_{\e\to0}\int_\R \frac{G_b(\e+\bi x-mb-nb\inv)G_b(Q+mb+nb\inv-2\e)}{G_b(Q+\bi x-\e)}f(x)dx\\
=&\sum_{\substack{kb+lb\inv < mb+nb\inv\\ k,l>0}} a_{kl}f(-\bi (kb+lb\inv)),}
where the constants $a_{kl}$ is the reside of the integrand at $-\bi (kb+lb\inv)$.
\end{Prop}

Finally we also need the new integral transformation obtained in \cite{Ip1}
\begin{Prop}\label{32} The 3-2 relation is given by
\Eq{\int_C G_b(\a+\bi \t)G_b(\b-\bi \t)G_b(\c-\bi \t)e^{-2\pi \bi (\b-\bi \t)(\c-\bi \t)}d\t = G_b(\a+\c)G_b(\a+\b),}
where the contour goes along $\R$ and separates the poles for $\bi \t$ and $-\bi \t$. By the asymptotic properties for $G_b$, the integral converges for $Re(\a-\b-\c)<\frac{Q}{2}$.
\end{Prop}

We will also need another important variant of the quantum dilogarithm.
\begin{Def} The function $g_b(x)$ is defined by
\Eq{g_b(x)=\frac{\over[\ze_b]}{G_b(\frac{Q}{2}+\frac{\log x}{2\pi \bi b})},}
where $\log$ takes the principal branch of $x$.
\end{Def}

\begin{Lem}\label{FT} \cite[(3.31), (3.32)]{BT} We have the following Fourier transformation formula:
\Eq{\int_{\R+\bi 0} \frac{e^{-\pi \bi  t^2}}{G_b(Q+\bi t)}X^{\bi b\inv t}dt=g_b(X), }
\Eq{\int_{\R+\bi 0} \frac{e^{-\pi Qt}}{G_b(Q+\bi t)}X^{\bi b\inv t}dt=g_b^*(X) ,} where $X$ is a positive operator and the contour goes above the pole
at $t=0$.
\end{Lem}

We will also need the following properties of $g_b(x)$.

\begin{Lem}\label{unitary} By \eqref{gb1}, $|g_b(x)|=1$ when $x\in\R_+$, hence $g_b(X)$ is a unitary operator for any positive operator $X$. Furthermore, by \eqref{selfdual} and Lemma \ref{FT}, we have the self-duality of $g_b(x)$ given by
\Eq{\label{selfdualg}g_b(X)=g_{b\inv}(X^{\frac{1}{b^2}}).}
\end{Lem}

\begin{Lem}\label{qsum}If $UV=q^2VU$ where $U,V$ are positive self adjoint operators, then
\begin{eqnarray}
\label{qsum0}g_b(U)g_b(V)&=&g_b(U+V),\\
\label{qsum1}g_b(U)^*Vg_b(U) &=& q\inv UV+V,\\
\label{qsum2}g_b(V)Ug_b(V)^*&=&U+q\inv UV.
\end{eqnarray}
Note that \eqref{qsum0} and \eqref{qsum1} together imply the pentagon relation
\Eq{\label{qpenta}g_b(V)g_b(U)=g_b(U)g_b(q\inv UV)g_b(V).}
\end{Lem}
If $UV=q^4VU$, then we apply the Lemma twice and obtain
\begin{eqnarray}
\label{qsum3}g_b(U)^*Vg_b(U)&=&V+[2]_qq^{2}VU+q^{4}VU^2,\\
\label{qsum4}g_b(V)Ug_b(V)^*&=&U+[2]_qq^{-2}UV+q^{-4}UV^2.
\end{eqnarray}
where $[2]_q=q+q\inv$.

As a consequence of the above Lemma, we also have the following
\begin{Lem}\cite{Ip3, Vo}\label{qbi} If $UV=q^2VU$ where $U,V$ are positive essentially self-adjoint operators, then $U+V$ is positive essentially self-adjoint, and
\Eq{\label{qbiEq} (U+V)^{\frac{1}{b^2}}=U^{\frac{1}{b^2}}+V^{\frac{1}{b^2}}.}
\end{Lem}

\subsection{Universal $R$ matrices for $\cU_{q}(\g)$}\label{sec:prelim:Rmatrix}
For $q:=e^h$, it is known \cite{D,J} that for the quantum group $\cU_h(\g)$ as a $\C[[h]]$-algebra completed in the $h$-adic topology, one can associate certain canonical, invertible element $R$ in an appropriate completion of $(\cU_h(\g))^{\ox 2}$ such that the the braiding relation and quasi-triangularity \eqref{br}-\eqref{qt2} are satisfied.

For the quantum groups $\cU_h(\g)$ associated to the simple Lie algebra $\g$, an explicit multiplicative formula has been computed independently in \cite{KR} and \cite{LS}, where the central ingredient involves the quantum Weyl group which induces Lusztig's isomorphism $T_i$. Explicitly, let 
\Eq{[U,V]_q:=qUV-q\inv VU} be the $q$-commutator.
\begin{Def}\cite{KR, Lu1}\label{Lus} Define
\Eq{T_i(K_j)=K_jK_i^{-a_{ij}}, \tab T_i(E_i)=-F_iK_i\inv, \tab T_i(F_i)=-K_iE_i,}
\Eq{T_i(E_j)=&(-1)^{a_{ij}}\frac{1}{[-a_{ij}]_{q_i}!}\left[\left[E_i,...[E_i,E_j]_{q_i^{\frac{a_{ij}}{2}}}\right]_{q_i^{\frac{a_{ij}+2}{2}}}...\right]_{q_i^{\frac{-a_{ij}-2}{2}}},\\
T_i(F_j)=&\frac{1}{[-a_{ij}]_{q_i}!}\left[\left[F_i,...[F_i,F_j]_{q_i^{\frac{a_{ij}}{2}}}\right]_{q_i^{\frac{a_{ij}+2}{2}}}...\right]_{q_i^{\frac{-a_{ij}-2}{2}}}.}
\end{Def}
Note that we have slightly modified the notations and scaling used in \cite{KR}.

\begin{Prop}\label{LuCox}\cite{Lu1, Lu2} $T_i$ satisfied the Weyl group relations:
\Eq{\underbrace{T_iT_jT_i...}_{-a_{ij}'+2} = \underbrace{T_j T_i T_j...}_{-a_{ij}'+2}.\label{coxeter},}
where $-a'_{ij}=\max\{-a_{ij},-a_{ji}\}$.
Furthermore, for $\a_i,\a_j$ simple roots, and an element $w=s_{i_1}...s_{i_k}\in W$ such that $w(\a_i)=\a_j$, we have 
\Eq{T_{i_1}...T_{i_k}(X_i)=X_j} for $X=E,F,K$.
\end{Prop}
\begin{Def}\cite{Ko} Define the (upper) quantum exponential function as
\Eq{Exp_{q}(x)=\sum_{k=0}^\oo \frac{z^k}{\lceil k\rceil_q!},}
where $\lceil k\rceil_q = \frac{1-q^k}{1-q}$, so that
\Eq{\lceil k\rceil_{q^2}!=[k]_q! q^{\frac{k(k-1)}{2}}.}
\end{Def}
\begin{Thm}\cite{KR, LS} Let $w_0=s_{i_1}...s_{i_N}$ be a reduced expression of the longest element. Then the universal $R$ matrix is given by
\Eq{\label{RRn}R=\bQ^\half \what{R}(i_N|s_{i_1}...s_{i_N-1})... \what{R}(i_2|s_{i_1})\what{R}(i_1)\bQ^\half,}
where 
\Eq{\bQ:=q^{\sum_{i,j=1}^n (dA\inv)_{ij}H_i\ox H_j},}
with $dA\inv$ the inverse of the symmetrized Cartan matrix $d_iA_{ij}$, and
\Eq{\label{RRnexp}\what{R}(i):=&Exp_{q_i^{-2}}((1-q_i^{-2})E_i\ox F_i),\\
\what{R}(i_l|s_{i_1}...s_{i_{l-1}}):=&(T_{i_1}\inv \ox T_{i_1}\inv)...(T_{i_{l-1}}\inv \ox T_{i_{l-1}}\inv)\what{R}(i_1).}
\end{Thm}

In both works \cite{KR, LS}, the expression for the $R$ matrix is obtained from the canonical element of the Drinfeld double of $\cU_h(\fb_+)$ generated by $E_i$'s and $H_i$'s. The Lusztig's isomorphism gives the ordered basis of $\cU_h(\fb_+)$, and there exists a dual pairing between $\cU_h(\fb_+)$ and $\cU_h(\fb_-)$ of this basis involving the quantum factorials $[k]_q!$, hence the expression \eqref{RRnexp}.

\subsection{Universal $R$ operator for $\cU_{q\til{q}}(\sl(2,\R))$}\label{sec:prelim:Rsl2}
In the case of $\cU_{q\til{q}}(\sl(2,\R))$, an expression of the $R$-operator is computed in \cite{BT}. It is given formally by
\Eq{R=q^{\frac{H\ox H}{4}} g_b(\be\ox \bf)q^{\frac{H\ox H}{4}},\label{RR}}
where we recall
\Eq{\be:=2\sin(\pi b^2)E,\tab \bf:=2\sin(\pi b^2)F, \tab K:=q^H.}
The operator $R$ acts naturally on $P_{\l_1}\ox P_{\l_2}$ by means of the positive representation. Note that the remarkable fact about this expression is the argument $\be\ox \bf$ inside the quantum dilogarithm $g_b$ which makes the expression a well-defined operator. In fact it is clear that $R$ acts as a unitary operator by Lemma \ref{unitary} of the properties of  $g_b(x)$. Furthermore, by the transcendental relations \eqref{transdef} and self-duality \eqref{selfdualg} of $g_b$, the expression \eqref{RR} is invariant under the change of $b\corr b\inv$:
\Eq{R=\til{R}:=\til{q}^{\frac{\til{H}\ox \til{H}}{4}} g_{b\inv}(\til{\be}\ox \til{\bf})\til{q}^{\frac{\til{H}\ox \til{H}}{4}}.}
Hence in fact it simultaneously serves as the $R$ operator of the modular double $\cU_{q\til{q}}(\sl(2,\R))$.

The properties as an $R$ operator implies certain functional equation for the quantum dilogarithm $g_b$. While the quasi-triangular relations \eqref{qt1}-\eqref{qt2} are equivalent to \eqref{qsum0}, the braiding relation 
$$\D'(X)R=R\D(X)$$
implies the following

\begin{Lem} \label{rank1prop}We have
\Eq{
\D'(X)R&=R\D(X)\nonumber\\
\iff\D'(\be)q^{\frac{1}{4}H\ox H}g_b(\be\ox \bf)q^{\frac{1}{4}H\ox H}&=q^{\frac{1}{4}H\ox H}g_b(\be\ox \bf)q^{\frac{1}{4}H\ox H}\D(\be)\nonumber\\
\iff(\be\ox K\inv +1\ox \be)g_b(\be\ox \bf)&=g_b(\be\ox \bf)(\be\ox K+1\ox \be),
}
and similarly
\Eq{
(\bf\ox 1 +K\ox \bf)g_b(\be\ox \bf)&=g_b(\be\ox \bf)(\bf\ox 1+K\inv\ox \bf).
}
\end{Lem}

\section{Generalized pentagon relations for $g_b(x)$}\label{sec:qdlog}
It turns out the exponential and pentagon relations \eqref{qsum0}-\eqref{qpenta} are not enough to show the properties of the universal $R$ matrix. In this section, following techniques from \cite{Ko}, we derive more general functional equations for $g_b(x)$ which generalize the pentagon relation as well as the quantum exponential relation.
\subsection{Simply-laced case}\label{sec:qdlog:ADE}
\begin{Prop}\label{genpenta} Let $U,V$ be positive self-adjoint operators such that
$c=\frac{UV-VU}{q-q\inv}$ is positive self-adjoint, and $Uc=q^2cU, Vc=q^{-2}cV$.
Then
\Eq{g_b(V)Ug_b^*(V)=&U+c,\\
g_b^*(U)Vg_b(U)=&c+V,
}
which also implies
\Eq{g_b(V)g_b(U)=g_b(U)g_b(c)g_b(V).\label{genpentagon}}
\end{Prop}
Note that if $UV=q^2VU$, these reduce to the usual pentagon relations \eqref{qsum1}-\eqref{qpenta}.
\begin{proof}
By induction, we calculate formally
\Eqn{
VU=&UV-(q-q\inv)c,\\
V^nU=&V^{n-1}UV-(q-q\inv)V^{n-1}c\\
=&V^{n-2}UV^2-(q-q\inv)V^{n-2}cV+(q-q\inv)V^{n-1}c\\
=&...\\
=&UV^n-(q-q\inv)(q^{2-2n}+q^{4-2n}+...+1)cV^{n-1}\\
=&UV^n-q(1-q^{-2n})cV^{n-1}\\
=&UV^n+q(1-q^{2n})cq^{-2n}V^{n-1}.
}
Hence by virtue of functional calculus, we can replace the power by complex powers $\bi b\inv t$, and apply the integration formula for $g_b(x)$. We obtain
\Eqn{
g_b(V)U=&Ug_b(V)+qc\int_{\R+\bi 0} (1-q^{2\bi b\inv t})q^{-2\bi b\inv t}e^{-\pi \bi t^2}\frac{V^{\bi b\inv t-1}}{G_b(Q+\bi t)}dt\\
=&Ug_b(V)+qc\int_{\R+\bi 0} (1-e^{-2\pi b(t-\bi b)})e^{2\pi b(t-\bi b)}e^{-\pi \bi (t-\bi b)^2}\frac{V^{\bi b\inv t}}{G_b(Q+\bi t+b)}dt\\
=&Ug_b(V)+qc\int_{\R+\bi 0} \frac{(1-e^{-2\pi b(t-\bi b)})e^{2\pi bt}q^{-2}e^{-2\pi bt}q }{(1-e^{2\pi \bi b(Q+\bi t)})}e^{-\pi \bi t^2}\frac{V^{\bi b\inv t}}{G_b(Q+\bi t)}dt\\
=&(U+c)g_b(V).}
Hence
$$g_b(V)Ug_b^*(V)=U+c$$
and
$$g_b(V)g_b(U)g_b^*(V)=g_b(U+c)=g_b(U)g_b(c).$$
Similarly, we also have
$$g_b^*(U)Vg_b(U)=c+V.$$
\end{proof}
\subsection{Non-simply-laced case}\label{sec:qdlog:BCFG}
In the non-simply-laced case, more $q$-commutators are involved. By applying the same techniques in the previous subsections repeatedly, we have the following relations.
\begin{Prop}\label{UVcd}
Let $U,V$ be positive operators and define $c,d$ to be
$$c=\frac{[U,V]}{q-q\inv},\tab d=\frac{q\inv c V-qVc}{q^2-q^{-2}},$$
such that $c,d$ are positive, and the following relations hold:
$$Uc=q^4 cU,\tab cd=q^4 dc,\tab dV=q^4 Vd.$$
Then we have
\Eq{g_b(V)Ug_b^*(V)=U+c+d.}
Similarly, we have
\Eq{g_b^*(U)Vg_b(U)=d'+c+V,}
where $$d'=\frac{q\inv Uc-q cU}{q^2-q^{-2}},$$ with
$$Vc=q^{-4}cV, \tab cd'=q^{-4}d'c, \tab d'U=q^{-4}Ud'.$$

Note that when $UV=q^4VU$, these reduce to the relation \eqref{qsum3}-\eqref{qsum4}.
\end{Prop}
Even more generally for the type $G_2$ case, by defining $e=\frac{q^{-2}dV-q^2Vd}{q^3-q^{-3}}$ such that $e$ is positive and $$Uc=q^6cU, \tab cd=q^6dc, \tab de=q^6ed, \tab eV=q^6Ve,$$ we have
$$g_b(V)Ug_b^*(V)=U+c+d+e.$$ Similarly relations hold for the other $q$-commutators $d'$ and $e'$.

Finally, we have the following useful functional relations generalizing the $q$-exponential relation.
\begin{Prop}\label{genexp}
Let $U,c,d,d'$ be as in Proposition \ref{UVcd}. Let $q=e^{\pi \bi b_s^2}$ and $q^2=e^{\pi \bi b_l^2}$. Then we have
\Eq{
g_{b_s}(U+c)=&g_{b_s}(U)g_{b_l}(d')g_{b_s}(c),\\
g_{b_l}(U+c+d)=&g_{b_l}(U)g_{b_s}(c)g_{b_l}(d).
}
\end{Prop}
Using Proposition \ref{genpenta} and \ref{UVcd}, these two relations are related by the transcendental relations in virtue with the approach in \cite{Ip3}, where the long roots and short roots are interchanged, and the self-duality \eqref{selfdualg} of $g_b(x)$.

We will leave the analogue of these functional relations in the case of type $G_2$ to the interested reader.

\section{Quantum Weyl element and Lusztig's isomorphism}\label{sec:QWeyl}
The starting point of the present work is the observation of the positivity appearing in the generators $\be_{ij}$ corresponding to the non-simple roots $\a_i+\a_j$. They are given by composition of certain $q$-commutators of simple root generators $\be_i$ and $\be_j$, and in turn is given by the Lusztig's isomorphism. Therefore to prove positivity, we show that Lusztig's isomorphism can actually be implemented by conjugations of certain elements $w_i$, which is known as the quantum Weyl element. In the compact case this is done in \cite{KR} and \cite{LS} by means of semi-simplicity of $\cU_q(\sl_2)$-modules in $\cU_q(\g)$-modules. In the current paper, we show that the $w_i$ can actually be implemented as a unitary operator, hence preserving positivity. The construction requires explicit calculation of the ribbon element $u$ and $v$ in Section \ref{sec:QWeyl:ribbon}, as well as the branching rules of $\cU_{q\til{q}}(\sl(2,\R))\subset\cU_{q\til{q}}(\g_\R)$ as positive representations in Section \ref{sec:QWeyl:branching} since we no longer have obvious semi-simplicity.
\subsection{Positivity of $\be_{ij}$}\label{sec:QWeyl:pos}
It is well-known that the Lusztig's isomorphism $T_i$ defined in Definition \ref{Lus} essentially gives the generators of the canonical basis of $\cU_q(\g)$. In the present case of positive representations, we also require the generators to be positive essentially self-adjoint. 

In the simply-laced case, we observe the following:
\begin{Prop}\label{eij} Fixed a positive representation $\cP_\l$. Then
\Eq{\be_{ij}:=\frac{[\be_j,\be_i]_{q^\half}}{q-q\inv}=\frac{q^{\half}\be_j\be_i-q^{-\half}\be_i\be_j}{q-q\inv}}
is positive essentially self-adjoint, and also satisfies the transcendental relations
\Eq{\til{\be_{ij}}:=\be_{ij}^{\frac{1}{b^2}}=\frac{\til{q}^{\half}\til{\be_j}\til{\be_i}-\til{q}^{-\half}\til{\be_i}\til{\be_j}}{\til{q}-\til{q}\inv}.}
\end{Prop}
\begin{proof} Without loss of generality, we can choose $w_0=w's_js_is_j$. Then it suffices to look at the representation in the case of type $A_2$ given by Proposition \ref{typeA2}. We obtain
\Eq{\be_{ij}=&e^{-\pi b(v-2w+2p_v+2p_w)}+e^{-\pi b(v+2p_v+2p_w)}+\\
&\tab +e^{-\pi b(u-w+2p_u+2p_v)}+e^{-\pi b(-u-w+2p_u+2p_v)},}
which is evidently positive. Since each term $q^2$ commute with the terms on its right, by Lemma \ref{qbi}, the operator is essentially self-adjoint, and satisfy the transcendental relation.
\end{proof}
We have similar observations in the non-simply-laced as well. Again it suffices to consider rank 2 case.
\begin{Prop}\label{eij2} In general, define the operators
\Eq{\be_{ij}=&(-1)^{a_{ij}}\left[\left[\be_i,...[\be_i,\be_j]_{q_i^{\frac{a_{ij}}{2}}}\right]_{q_i^{\frac{a_{ij}+2}{2}}}...\right]_{q_i^{\frac{-a_{ij}-2}{2}}}\prod_{k=1}^{-a_{ij}}(q_i^k-q_i^{-k})\inv.\label{eij3}}
Then it is positive essentially self-adjoint, and satisfy the generalized transcendental relations, where $\be_{ij}^{\frac{1}{b_i^2}}$ is given by the same expression as $\be_{ij}$ with all $\be_i$ replaced by $\til{\be_i}$, $q_i$ replaced by $\til{q_i}$, and $a_{ij}$ replaced by $a_{ji}$.
\end{Prop}
\begin{proof} These are calculated directly from the explicit expression of the positive representation of type $B_2, G_2$ and also the transcendental relation using expressions of type $C_2$ given in \cite{Ip3}.
\end{proof}

We note that $\be_{ij}=T_i(\be_j)$ up to some constant. Therefore if we can show that $T_i$ are given by \emph{inner automorphism} of some unitary element, then both positivity and transcendental relations of the remaining generators in higher rank will be immediate. This is achieved by the use of the quantum Weyl elements described in Section \ref{sec:QWeyl:Weyl}.

Finally we define $\bf_{ij}$ with the exact same formula \eqref{eij3} with $\be$ replaced by $\bf$. Then using the Weyl element $w_0$ derived in Section \ref{sec:QWeyl:Weyl} we see that it also satisfies all the properties enjoyed by $\be_{ij}$.
\subsection{Calculation of the ribbon element $u$ and $v$ for $\cU_{q\til{q}}(\sl(2,\R))$}\label{sec:QWeyl:ribbon}
In this section we restrict the attention to a fixed positive representation $P_\l$ of $\cU_{q\til{q}}(\sl(2,\R))$. Let the $R$ operator be given by \eqref{RR}. Explicitly, it can be written as
\Eq{\bQ^\half \left(\int_{\R+\bi 0}\frac{e^{-\pi \bi t^2}\be^{\bi b\inv t}\ox \bf^{\bi b\inv t}}{G_b(Q+\bi t)}dt \right)\bQ^\half,\label{Rbasis1}}
where 
\Eq{\bQ=q^{\frac{H\ox H}{2}}=\sum_{n=0}^\oo \left(\frac{\pi \bi b^2}{2}\right)^n\frac{H^n\ox H^n}{n!},\label{Rbasis2}}
which we will write it informally as $R=\sum_k \a_k\ox \b_k$.

We wish to calculate the element 
\Eq{u=m^{op}\circ (1\ox S)R=\sum_k S(\b_k)\a_k,} 
which is crucial in the analysis of quasi-triangular Hopf algebra. Here we will first calculate the expression formally using an extension of the antipode $S$. In Section \ref{sec:Cstar}, we will then define $u$ rigorously as an element in certain multiplier Hopf-* algebra.

From the expression of $u$, it means we need to calculate the action of $\bf^{\bi b\inv t}\be^{\bi b\inv t}$, in other words we need to calculate the action of $\be^{\bi b\inv t}$ and $\bf^{\bi b\inv t}$ under the positive representation $\cP_\l$ of $\cU_{q\til{q}}(\sl(2,\R))$. Furthermore we also need the actual effect of the antipode. We introduce the expression
\Eq{S(\be)=e^{\pi \bi bQ}\be=-q\be,\tab S(\bf)=e^{-\pi \bi bQ}\bf=-q\inv \bf, \tab S(H)=-H. \label{SS}}
consistent with the usual definition, and define $S$ on the complex powers by
\Eq{S(\be^{\bi b\inv t}):=e^{-\pi Qt}\be^{\bi b\inv t}, \tab S(\bf^{\bi b\inv t}):=e^{\pi Qt}\bf^{\bi b\inv t}.}
Again the definition is rigorous once we impose the setting of multiplier Hopf-* algebra in Section \ref{sec:Cstar}.

\begin{Lem} The action of $\be^{\bi b\inv t}$ and $\bf^{\bi b\inv t}$ on $f(x)$ is given by
\Eq{\be^{\bi b\inv t}\cdot f(x)=&e^{\pi \bi (x-\l) t}e^{\frac{-\pi \bi t^2}{2}}\frac{G_b(\frac{Q}{2}+\bi x-\bi \l)}{G_b(\frac{Q}{2}+\bi x-\bi \l-\bi t)}f(x-t),\\
\bf^{\bi b\inv t}\cdot f(x)=&e^{\pi \bi (x+\l) t}e^{\frac{\pi \bi t^2}{2}}\frac{G_b(\frac{Q}{2}+\bi x+\bi \l+\bi t)}{G_b(\frac{Q}{2}+\bi x+\bi \l)}f(x+t).}
Note that these actions are unitary transformations.
\end{Lem}
\begin{proof}
\Eqn{
\be &= e^{-\pi b x+\pi b\l-2\pi b p}+e^{\pi b x-\pi b\l-2\pi b p}\\
&=g_b^*(e^{-2\pi b(x-\l)})e^{\pi bx-\pi b\l-2\pi bp}g_b(e^{-2\pi b(x-\l)}\\
&=g_b^*(e^{-2\pi b(x-\l)})e^{\pi \bi (\frac{x^2}{2}-\l x)}e^{-2\pi bp}e^{-\pi \bi (\frac{x^2}{2}-\l x)}g_b(e^{-2\pi b(x-\l)}),\\
\be^{\bi  b\inv t}\cdot f(x)&=g_b^*(e^{-2\pi b(x-\l)})e^{\pi \bi (\frac{x^2}{2}-\l x)}e^{-2\pi \bi tp}e^{-\pi \bi (\frac{x^2}{2}-\l x)}g_b(e^{-2\pi b(x-\l)})\cdot f(x)\\
&=g_b^*(e^{-2\pi b(x-\l)})e^{\pi \bi (\frac{x^2}{2}-\l x)}e^{-\pi \bi (\frac{(x-t)^2}{2}-\l (x-t))}g_b(e^{-2\pi b(x-\l-t)})f(x-t)\\
&=e^{\pi \bi (x-\l) t}e^{\frac{-\pi \bi t^2}{2}}\frac{g_b(e^{-2\pi b(x-\l-t)})}{g_b(e^{-2\pi b(x-\l)})}f(x-t)\\
&=e^{\pi \bi (x-\l) t}e^{\frac{-\pi \bi t^2}{2}}\frac{G_b(\frac{Q}{2}+\bi x-\bi \l)}{G_b(\frac{Q}{2}+\bi x-\bi \l-\bi t)}f(x-t).
}

Similarly,
\Eqn{
\bf &= e^{-\pi b x-\pi b\l+2\pi b p}+e^{\pi b x+\pi b\l+2\pi b p}\\
&=g_b(e^{-2\pi b(x+\l)})e^{\pi bx+\pi b\l+2\pi bp}g_b^*(e^{-2\pi b(x+\l)}\\
&=g_b(e^{-2\pi b(x+\l)})e^{-\pi \bi (\frac{x^2}{2}+\l x)}e^{2\pi bp}e^{\pi \bi (\frac{x^2}{2}+\l x)}g_b^*(e^{-2\pi b(x+\l)}),\\
\bf^{\bi  b\inv t}\cdot f(x)&=g_b(e^{-2\pi b(x+\l)})e^{-\pi \bi (\frac{x^2}{2}+\l x)}e^{2\pi \bi tp}e^{\pi \bi (\frac{x^2}{2}+\l x)}g_b^*(e^{-2\pi b(x+\l)})\cdot f(x)\\
&=g_b(e^{-2\pi b(x+\l)})e^{-\pi \bi (\frac{x^2}{2}+\l x)}e^{\pi \bi (\frac{(x+t)^2}{2}+\l (x+t))}g_b^*(e^{-2\pi b(x+\l+t)})f(x+t)\\
&=e^{\pi \bi (x+\l) t}e^{\frac{\pi \bi t^2}{2}}\frac{g_b(e^{-2\pi b(x+\l)})}{g_b(e^{-2\pi b(x+\l+t)})}f(x+t)\\
&=e^{\pi \bi (x+\l) t}e^{\frac{\pi \bi t^2}{2}}\frac{G_b(\frac{Q}{2}+\bi x+\bi \l+\bi t)}{G_b(\frac{Q}{2}+\bi x+\bi \l)}f(x+t).
}
\end{proof}
Hence combining, we have
\Eqn{
\bf^{\bi b\inv t}\be^{\bi b\inv t}&=e^{\pi \bi (x+\l) t}e^{\frac{\pi \bi t^2}{2}}\frac{G_b(\frac{Q}{2}+\bi x+\bi \l+\bi t)}{G_b(\frac{Q}{2}+\bi x+\bi \l)}e^{\pi \bi (x+t-\l) t}e^{\frac{-\pi \bi t^2}{2}}\frac{G_b(\frac{Q}{2}+\bi x-\bi \l+\bi t)}{G_b(\frac{Q}{2}+\bi x-\bi \l)}f(x)\\
&=e^{\pi \bi t(2x+t)}\frac{G_b(\frac{Q}{2}+\bi x+\bi \l+\bi t)}{G_b(\frac{Q}{2}+\bi x+\bi \l)}\frac{G_b(\frac{Q}{2}+\bi x-\bi \l+\bi t)}{G_b(\frac{Q}{2}+\bi x-\bi \l)}f(x).\\
}

\begin{Thm} $u=\sum S(\b_k)\a_k = m^{op}(1\ox S) R$ is given by
\Eq{u=e^{2\pi \bi (\l^2+\frac{Q^2}{4})}K^\frac{Q}{b}.}
\end{Thm}
\begin{proof}
First note that $H\be=\be H+2\be$ implies
$$H^n \be = \be(H+2)^n$$
$$H^n \be^{\bi b\inv t}=\be^{\bi b\inv t}(H+2\bi b\inv t)^n.$$
Similarly
$$H^n \bf^{\bi b\inv t}=\bf^{\bi b\inv t}(H-2\bi b\inv t)^n.$$
Note that $H$ commute with $\bf^{\bi b\inv t}\be^{\bi b\inv t}$.

Hence using the ``continuous basis" \eqref{Rbasis1}-\eqref{Rbasis2}
\Eqn{
&(H^n\ox H^n) (e^{\bi b\inv t}\ox f^{\bi b\inv t}) (H^m\ox H^m)\\
&=H^n e^{\bi b\inv t}H^m \ox H^n f^{\bi b\inv t}H^m,\\
m^{op}(1\ox S)&=S(H^n \bf^{\bi b\inv t}H^m)H^n \be^{\bi b\inv t}H^m\\
&=(-1)^m(-1)^n H^m e^{\pi  Qt}\bf^{\bi b\inv t}H^nH^n \be^{\bi b\inv t}H^m\\
&=(-1)^m(-1)^n H^m e^{\pi  Qt}\bf^{\bi b\inv t} \be^{\bi b\inv t}(H+2\bi b\inv t)^{2n}H^m\\
&=e^{\pi Qt} \bf^{\bi b\inv t} \be^{\bi b\inv t}(-H^2)^m(-H^2-4\bi b\inv t H+4b^{-2}t^2)^n.
}
Hence
\Eqn{
m^{op}(1\ox S)R&=\left(\int_{\R+\bi 0} \frac{e^{-\pi \bi t^2+\pi Qt }}{G_b(Q+\bi t)} \bf^{\bi b\inv t} \be^{\bi b\inv t} K^{-\bi b\inv t}e^{\pi \bi t^2}dt\right) q^{-\frac{H^2}{2}}\\
&=\left(\int_{\R+\bi 0} \frac{e^{\pi Qt}}{G_b(Q+\bi t)} \bf^{\bi b\inv t} \be^{\bi b\inv t} K^{-\bi b\inv t}dt \right)q^{-\frac{H^2}{2}},
}
and the action on $f(x)$ is given by ($K=e^{-2\pi bx}=q^{2\bi b\inv x}$, so $H=2\bi b\inv x$):

\Eqn{
u=&\int_{\R+\bi 0}e^{-(\pi \bi b^2)(2\bi b\inv x)^2/2}\cdot e^{-2\pi b x(-\bi b\inv t)}e^{\pi \bi t(2x+t)}\cdot\\
&\frac{G_b(\frac{Q}{2}+\bi x+\bi \l+\bi t)}{G_b(\frac{Q}{2}+\bi x+\bi \l)}\frac{G_b(\frac{Q}{2}+\bi x-\bi \l+\bi t)}{G_b(\frac{Q}{2}+\bi x-\bi \l)}\frac{e^{\pi Qt}}{G_b(Q+\bi t)}dt\\
=&\frac{\int_{\R+\bi 0}e^{2\pi \bi (x+t)^2+2\pi Qt}G_b(\frac{Q}{2}+\bi x+\bi \l+\bi t)G_b(\frac{Q}{2}+\bi x-\bi \l+\bi t)G_b(-\bi t)dt}{G_b(\frac{Q}{2}+\bi x+\bi \l)G_b(\frac{Q}{2}+\bi x-\bi \l)}\\
=&e^{2\pi \bi (\l^2+\frac{Q^2}{4})}e^{-2\pi Qx}\cdot\\
&\frac{\int_{\R+\bi 0}e^{-2\pi \bi (\frac{Q}{2}+\bi x+\bi \l+it)(\frac{Q}{2}+\bi x-\bi \l+\bi t)}G_b(\frac{Q}{2}+\bi x+\bi \l+\bi t)G_b(\frac{Q}{2}+\bi x-\bi \l+\bi t)G_b(-\bi t)dt}{G_b(\frac{Q}{2}+\bi x+\bi \l)G_b(\frac{Q}{2}+\bi x-\bi \l)}\\
&=e^{2\pi \bi (\l^2+\frac{Q^2}{4})}K^{\frac{Q}{b}},
}
where in the last line we used the 3-2 relations from Proposition \ref{32}.
\end{proof}

\begin{Rem}
Letting $l=-\frac{Q}{2}+\bi \l$, one should compare this expression 
\Eq{u=q^{-2\frac{l}{b}(\frac{l}{b}+\frac{Q}{b})}K^\frac{Q}{b},} with 
with the expression from the compact case \cite{KR} on the $2j+1$ dimensional module $V_j$:
\Eq{u=q^{-2j(j+1)}K.}
\end{Rem}

Now one can check that the following is satisfied: $S^2(a)=uau\inv$:
\Eqn{
S^2(\be^{\bi b\inv t})&=e^{-2\pi Qt}\be^{\bi b\inv t}=K^{\frac{Q}{b}}\be^{\bi b\inv t}K^{-\frac{Q}{b}}=u\be^{\bi b\inv t}u\inv,\\
S^2(\bf^{\bi b\inv t})&=e^{2\pi Qt}\bf^{\bi b\inv t}=K^{\frac{Q}{b}}\bf^{\bi b\inv t}K^{-\frac{Q}{b}}=u\bf^{\bi b\inv t}u\inv.
}
\begin{Def} The ribbon element $v$ is defined to be the constant operator acting on $P_\l$ as multiplication by
\Eq{v=e^{2\pi \bi (\l^2+\frac{Q^2}{4})},}
such that $u=vK^{\frac{Q}{b}}$.
\end{Def}
\subsection{Branching rules for $\cU_{q\til{q}}(\sl(2,\R))\subset \cU_{q\til{q}}(\g_\R)$}\label{sec:QWeyl:branching}
In the work of \cite{KR} and \cite{LS}, the quantum Weyl element is defined by decomposing $\cU_q(\g)$ into irreducible $\cU_q(\sl(2))$ submodules, which exists because the algebra involved are semisimple. In the current setting of positive representations, which is infinite dimensional, it is not at all clear whether the same decomposition is possible. It turns out that the branching rules is particularly simple, and remarkably it resembles both the decomposition of the tensor product representation $P_\a\ox P_\b$ (cf. \cite{PT2}) and the Peter-Weyl type decomposition of $L^2(SL_q^+(2,\R))$ (cf. \cite{Ip1}) with exactly the same Plancherel measure $|S_b(Q+2\a)|^2d\a$:

Let $q_i=e^{\pi \bi b_i^2}$ and $Q=b_i+b_i\inv$. 
\begin{Thm} \label{branchingrules} Fixed any positive representation $\cP_\l\simeq L^2(\R^N)$ of $\cU_{q\til{q}}(\g_\R)$, where $N=l(w_0)$. As a representation of $\cU_{q_i\til{q_i}}(\sl(2,\R))\subset \cU_{q\til{q}}(\g_\R)$ corresponding to the root $\a_i$, 
\Eq{\cP_{\l}\simeq L^2(\R^{N-2})\ox \int_{\R_+} P_\c d\mu(\c)}
is a unitary equivalence, where $P_\c$ is the positive representation of $\cU_{q_i\til{q_i}}(\sl(2,\R))$ with parameter $\c$, and the Plancherel measure is given by \Eq{d\mu(\c)=|S_{b_i}(Q_i+2\c)|^2d\c,} where $S_b(x)=G_b(x)e^{\frac{\pi \bi }{2}x(Q-x)}$.
\end{Thm}
\begin{proof} Using the same techniques as in \cite{Ip1}, it suffices to diagonalize the Casimir element.

By taking $w_0=w's_i$ so that the simple reflection for root $i$ appears on the right, the action of $\be_i$ is the standard action (using notation from Section \ref{sec:prelim:pos})
$$\be_i=[u_i^1]e(-p_i^1),$$
while the action of $\bf_i$ and $K_i$ is given by \eqref{FF} and \eqref{KK} respectively.

Note that $\be_i$ commute with the terms of $\bf_i$ for $k>1$. Then the rescaled Casimir element $\bc_i$ for this $\cU_{q_i\til{q_i}}(\sl(2,\R))$ representation
\Eq{\bc_i:=\left(\frac{\bi }{q_i-q_i\inv}\right)^{-2}C_i = \bf_i\be_i-(q_iK_i+q_i\inv K_i\inv)} is given by
\Eqn{
&\sum_{k=1}^n\left[-\sum_{j=1}^{v(i,k)-1} a_{i,r(j)}v_j-u_i^k-2\l_i\right]e(p_{i}^{k})[u_i^1]e(-p_i^1)-(q_iK+q_i\inv K\inv)\\
=&\sum_{k=2}^n\left[-\sum_{j=1}^{v(i,k)-1} a_{i,r(j)}v_j-u_i^k-2\l_i\right][u_i^1]e(p_{i}^{k}-p_i^1)+2\cosh\left(\pi\bb\cdot (\sum_{j=1}^{v(i,1)-1}a_{i,r(j)}v_j+2\l_i)\right).
}
Here we used the notation $(\bb\cdot-)$ so that variables corresponding to short (resp. long) root get multiplied by $b_s$ (resp. $b_l$) (cf. Definition \ref{usul}).  Applying the transformation by multiplication by $g_{b_i}^*(2u_i^1)$ will eliminate the $[u_i^1]$ factor:
\Eqn{
\simeq&\sum_{k=2}^n\left[-\sum_{j=1}^{v(i,k)-1} a_{i,r(j)}v_j-u_i^k-2\l_i\right]e(p_{i}^{k}-p_i^1)+2\cosh\left(\pi \bb\cdot(\sum_{j=1}^{v(i,1)-1}a_{i,r(j)}v_j+2\l_i)\right).}

Now we know from previously that the terms from $k=2$ to $k=n$ $q_i^2$-commute successively. Hence there exists transformation by certain $g_{b_i}(x)$ where the argument is given by the differences of the factors, that the above is unitary equivalent to first term:
\Eqn{
\simeq2\cosh\left(\pi \bb\cdot(\sum_{j=1}^{v(i,1)-1}a_{i,r(j)}v_j+2\l_i)\right)+e^{\pi \bb\cdot(-\sum_{j=1}^{v(i,n)-1}a_{i,r(j)}v_j-u_i^n-2\l_i+2p_i^n-2p_i^1)}.}
Now apply the transformation $u_i^1\mapsto u_i^1-u_i^n$ to get rid of $p_i^1$.
Then apply $$p_i^n\mapsto p_i^n+\l_i-\half\left(\sum_{j=1}^{v(i,n)-1}a_{i,r(j)}v_j-u_i^n\right),$$
so that the last term becomes simply $e^{2\pi b_ip_i^n}$. Finally apply 
$$u_i^n\mapsto u_i^n-\l_i-\half\sum_{j=1, v_j\neq u_i^n}^{v(i,1)-1} a_{i,r(j)} v_j,$$ 
and we arrive at 
$$\bc_i\simeq e^{2\pi b_iu_i^n}+e^{-2\pi b_iu_i^n}+e^{2\pi b_ip_i^n},$$
and we know from \cite{Ip1, Ka} that this is unitary equivalent to
$$\bc_i\simeq\int_{\R^+} (e^{2\pi b_i\c}+e^{-2\pi b_i\c}) d\mu(\c),$$
with the measure $d\mu(\c)=|S_{b_i}(Q_i+2\c)|^2d\c$.

Finally by reversing the transformations above, skipping the variables $u_i^n$, we obtain an explicit expression of the action $\be_i, \bf_i$ involving only the last variable in $$L^2(\R^{N-2})\ox \int_{\R_+} P_\c d\mu(\c).$$
\end{proof}
\subsection{Unitary action of the Weyl element $w_i$}\label{sec:QWeyl:Weyl}
Following the compact case in \cite{KR}, we adjoin an element $w$ to $\cU_{q\til{q}}(\sl(2,\R))$ such that it satisfies the following:
\Eq{
w \be w\inv &= \bf,\\
w \bf w\inv &= \be,\\
w K w\inv &= K\inv,
}
with the Hopf algebra structure
\Eq{
\D w&=R\inv (w\ox w),\\
S(w)&=wK^{\frac{Q}{b}},\\
\e(w)&=1,
}
so that in addition it satisfies
\Eq{w^2=v=uK^{-\frac{Q}{b}},}
which implies
\Eq{S(w)w=u.}

On the positive representations considered in Proposition \ref{canonicalsl2}, we define the action of $w$ on $P_\l=L^2(\R)$ as a unitary operator
\Eq{w\cdot f(x)=e^{\pi \bi (\l^2+\frac{Q^2}{4})}f(-x),\label{wact}}
so that all the above properties are satisfied.

Now in the general case, consider the positive representation $\cP_\l$ of $\cU_{q\til{q}}(\g_\R)$. For each simple roots $\a_i$, using the branching rule of $\cU_{q\til{q}}(\sl(2,\R))$ from Theorem \ref{branchingrules}, we define the action of $w_i$ on $\cP_\l$ as
\Eq{w_i:=I_{N-2}\ox \int_{\R_+} w_i^\a d\mu(\a),\label{ww}}
where $w_i^\a$ acts as \eqref{wact} on $P_\a$. It is clear that $w_i$ is a unitary operator since the branching rule for $\cU_{q\til{q}}(\sl(2,\R))$ is a unitary equivalence.

Now we can follow the approach in [KR] and calculate the action of $w_i\be_jw_i\inv$ and $w_i\bf_jw_i\inv$, while we also have \Eq{w_iK_jw_i\inv=K_jK_i^{-a_{ij}}.} We will do the calculations mainly for $\be_i$, while those for $\bf_i$ is similar.
Let
\Eqn{
\over[\be]_i&:=q_i^{\half}K_i^{-\half}\be_i,\\
\over[\bf]_i&:=q_i^{\half}K_i^{\half}\bf_i,
}
so that the $R$ matrix can be expressed as
\Eq{R_i=g_b(\over[\be]_i\ox \over[\bf]_i)q^{\frac{H_i\ox H_i}{2}}.}
Note that $\over[\be]_i$ and $\over[\bf]_i$ are still positive essentially self-adjoint and satisfy the transcendental relations.

For any Hopf algebra $A$, one can define the adjoint action of $A$ on itself by
\Eq{a\circ b=\sum_i a^i b S(a_i),} where $\D(a)=\sum_i a^i\ox a_i$. Then the action $w_i\circ \over[\be]_j$ can be calculated exactly as in [KR], taking into account the new antipode, and we still obtain
\Eq{w_i\circ \over[\be]_j=w_i \over[\be]_j K_i^{\frac{1}{2}a_{ij}}w_i\inv.}
On the other hand, $V_{ij}=\{(\over[\be]_i)^n\circ \over[\be]_j\}_{n=0}^{-a_{ij}}$ is an irreducible $U_q(\sl(2))$-module with highest weight $-a_{ij}$. Since $w_i$ flips the action of $E_i$ and $F_i$ by definition, the adjoint action maps the lowest weight vector to highest weight vector. In particular, we have
\Eq{
w_i\circ \over[\be]_j&=c_{ij} \over[\be]_i^{-a_{ij}} \circ \over[\be]_j
}
for some constant $c_{ij}$. Note that this equation also holds for the modular double counterpart $\til{\be_j}$. Hence the constant $c_{ij}$ is uniquely determined by the fact that $w_i\be_j w_i\inv$ is positive and satisfy the transcendental relation.
Now it is easy to calculate that
\Eq{\over[\be]_i^{-a_{ij}}\circ \over[\be]_j = (-1)^{a_{ij}}\be_{ij}K_i^{\frac{a_{ij}}{2}}K_j^{-\half},}
where $\be_{ij}$ is defined in Proposition \ref{eij} and \ref{eij2}, and
\Eqn{
&w_i \over[\be]_j K_i^{\frac{a_{ij}}{2}}w_i\inv\\
&=w_i q_j^{\half}K_j^{-\half}\be_jK_i^{\frac{a_{ij}}{2}}w_i\inv\\
&=w_i \be_j w_i\inv q_j^{-\half}K_j^{-\half}.
}

Hence combining we have
\Eqn{
w_i \be_j w_i\inv =c_{ij}'\be_{ij}K_i^{\frac{a_{ij}}{2}.}
}
The constant can now be easily determined by positivity to be $c_{ij}'=q_i^{-\frac{a_{ij}^2}{4}}$.

Finally we define $w_i':=w_iq_i^{\frac{H_i^2}{4}}$ and $T_i(a)=w_i'a(w_i')\inv$.
We have
\begin{Thm}\label{LusThm} The action $T_i$ resembles the Lusztig's isomorphism \cite{Lu1}, while taking positivity into account. In particular Proposition \ref{LuCox} is still satisfied. We have
\Eq{
T_i(\be_i)=&q_i \bf_iK_i\inv=q_i\inv K_i\inv \bf_i,\\
T_i(\bf_i)=&q_i\inv K_i \be_i=q_i\be_i K_i,\\
T_i(\be_j)=&\be_{ij},\mbox{ for $i,j$ adjacent,}\\
T_i(\bf_j)=&\bf_{ij},\mbox{ for $i,j$ adjacent,}\\
T_i(K_j)=&K_jK_i^{-a_{ij}}.
}
\end{Thm}
\begin{proof}
\Eqn{
T_i(\be_i)&=w_iq_i^{\frac{H_i^2}{4}}\be_iq_i^{-\frac{H_i^2}{4}}w_i\inv\\
&=w_i\be_i q_i^{\frac{(H_i+2)^2}{4}}q_i^{-\frac{H_i^2}{4}}w_i\inv\\
&=w_i\be_i K_i q_i w_i\inv\\
&=q_i \bf_iK_i\inv=q_i\inv K_i\inv \bf_i,
}
\Eqn{
T_i(\be_j)&=w_iq_i^{\frac{H_i^2}{4}}\be_jq_i^{-\frac{H_i^2}{4}}w_i\inv\\
&=w_i\be_iq_i^{\frac{(H_i+a_{ij})^2}{4}}q_i^{-\frac{H_i^2}{4}}w_i\inv\\
&=q_i^{-\frac{a_{ij}^2}{4}}\be_{ij}K_i^{\frac{a_{ij}}{2}} K_i^{-\frac{a_{ij}}{2}}q_i^{\frac{a_{ij}^2}{4}}\\
&=\be_{ij},
}
and similarly for the calculations of $\bf$. The action $T_i$ only differs from Lusztig's isomorphism by certain scaling, hence Proposition \ref{LuCox} is still satisfied due to positivity that restricts the scaling.
\end{proof}
\begin{Cor}\label{Tpos} Under the positive representations $\cP_\l$, the operators $T_{i_1}...T_{i_k}\be_j$ are positive essentially self-adjoint, and satisfies the transcendental relations.
\end{Cor}

\section{Universal $R$ operator}\label{sec:R}
We are now in the position to define the universal $R$ operator in the flavor of Section \ref{sec:prelim:Rmatrix} and Section \ref{sec:prelim:Rsl2}, generalizing the respective formula. 

\begin{Thm}\label{mainR} Let $w_0=s_{i_1}s_{i_2}...s_{i_N}$ be a reduced expression of the longest element of the Weyl group. Then the universal $R$-operator for the positive representations of $\cU_{q\til{q}}(\g_\R)$ acting on $\cP_{\l_1}\ox \cP_{\l_2}\simeq L^2(\R^N)\ox L^2(\R^N)$ is given by
\Eq{R=&\prod_{ij}q_i^{\half(A\inv)_{ij}H_i\ox H_j}\prod_{k=1}^{N} g_b(\be_{\a_k}\ox \bf_{\a_k}) \prod_{ij}q_i^{\half(A\inv)_{ij}H_i\ox H_j},}
where $\be_{\a_k}=T_{i_1} T_{i_2}... T_{i_{k-1}} \be_{i_k}$, similarly for $\bf_{\a_k}$. The product is such that the term $k=1$ appears on the rightmost side.
It is clear that $R$ is a unitary operator.
\end{Thm}
\begin{Rem}\label{bar}
By Corollary \ref{Tpos}, the generator $\be_{\a_k}\ox \bf_{\a_k}$ are positive, hence the expression is well-defined. By commuting the last factor, $R$ can also be written as
\Eq{R=&\prod_{ij}q_i^{\half(A\inv)_{ij}H_i\ox H_j}\prod_{k=1}^{N} g_b(\over[\be]_{\a_k}\ox \over[\bf]_{\a_k}) ,}
where $\over[\be]_{\a_k}=q_{i_k}^{-\half}K_{\a_k}^{\half}\be_{\a_k}$ and $\over[\bf]_{\a_k}=q_{i_k}^{-\half}K_{\a_k}^{-\half}\bf_{\a_k}$.
\end{Rem}

By general theory developed in \cite{KR} and \cite{LS}, the $R$ operator can be written in terms of the root components as follows. Let $W_0=w_{i_1}...w_{i_N}$. Then
\Eqn{R\inv (W_0\ox W_0)&=\D(W_0)\\
=&\D(w_{i_1})...\D(w_{i_N})\\
=&R_{i_1}\inv (w_{i_1}\ox w_{i_1})...R_{i_N}\inv (w_{i_N}\ox w_{i_N}),
}
or
\Eq{R=(W_0\ox W_0)(w_{i_N}\ox w_{i_N})\inv R_{i_N}...(w_{i_1}\ox w_{i_1})\inv R_{i_1}.\label{wR}}
It turns out, not surprisingly, that it suffices to prove the braiding relations and quasi-triangularity relations in the case of rank=2.
It is known that the braiding relations and quasi-triangularity relations imply 
\Eq{(S\ox S)(R)=R,}
In rank=2 case, this means that the expression of $R$ corresponding to the Coxeter relation \eqref{coxeter} for the change of words of $w_0$ is the same. Therefore the definition given in Theorem \ref{mainR} does not depend on the choice of reduced expression, hence the expression of $R$ is uniquely defined.

In the next section, we will show that this $R$ operator arise as the canonical element of certain Drinfeld's double construction. Hence the braiding relation and quasi-triangularity will be automatic from the formal algebraic manipulation. However, it is still instructive to see explicitly how the functional equations of the quantum dilogarithm $g_b(x)$ play a role in the calculation of these properties.

\subsection{Braiding relations in simply-laced case}\label{sec:R:braiding}
Consider the case of type $A_2$, and choose $w_0=s_1s_2s_1$. The universal $R$-operator is given explicitly by
\Eq{R=\bQ^\half g_b(\be_2\ox \bf_2)g_b(\be_{12}\ox\bf_{12})g_b(\be_1\ox\bf_1)\bQ^\half ,}
where 
\Eq{\bQ=q^{\frac{2}{3}H_1\ox H_1+\frac{1}{3}H_1\ox H_2+\frac{1}{3}H_2\ox H_1+\frac{2}{3}H_2\ox H_2}.}
We will show that 
\Eq{\D'(\be_1)R=R\D(\be_1).} The other cases are similar.
\Eqn{
\D'(\be_1)\bQ^\half=&(\be_1\ox K_1^{-\half}+K_1^{\half}\ox \be_1)\bQ^\half\\
=&\bQ^\half(\be_1\ox K_1\inv+1\ox \be_1).
}
Next we have
$$(\be_1\ox K_1\inv+1\ox \be_1)g_b(\be_2\ox \bf_2)=g_b(\be_2\ox \bf_2)(\be_1\ox K_1\inv+\be_{12}\ox q^{\half}K_1\inv \bf_2+1\ox \be_1),$$
where we used the generalized pentagon relation \eqref{genpentagon} with
$$\frac{[\be_2\ox \bf_2, \be_1\ox K_1\inv]}{q-q\inv}=\be_{12}\ox q^{\half}K_1\inv \bf_2,$$ and the fact that $1\ox \be_1$ commute with $\be_2\ox \bf_2$.
Then we have by \eqref{genpentagon} again
$$(\be_1\ox K_1\inv+\be_{12}\ox q^{\half}K_1\inv \bf_2+1\ox \be_1)g_b(\be_{12}\ox \bf_{12})=g_b(\be_{12}\ox \bf_{12})(\be_1\ox K_1\inv+1\ox \be_1),$$
where we used 
$\be_1\bf_{!2}=\bf_{12}\be_1+q^{\half}(q-q\inv) K_1\inv\bf_2$
such that
$$\frac{[1\ox \be_1, \be_{12}\ox \bf_{12}]}{q-q\inv}=\be_{12}\ox q^{\half}K_1\inv \bf_2.$$
Finally, by Lemma \ref{rank1prop},
$$(\be_1\ox K_1\inv+1\ox \be_1)g_b(\be_1\ox \bf_1)=g_b(\be_1\ox \bf_1)(\be_1\ox K_1+1\ox \be_1),$$
and
\Eqn{
(\be_1\ox K_1+1\ox \be_1)\bQ^\half=&\bQ^\half(\be_1\ox K_1^{\half}+K_1^{-\half}\ox \be_1)\\
=&\bQ^\half\D(E_1).
}

Recall the expression of $R$ given by \eqref{wR}. What we have shown is that (by abuse of notation, write $w_i:=w_i\ox w_i$):
\Eqn{
W_0 w_1\inv R_1w_2\inv R_2w_1\inv R_1\D(E_1)=&\D'(E_1)W_0 w_1\inv R_1w_2\inv R_2w_1\inv R_1,\\
W_0 w_2\inv R_2w_1\inv R_1w_2\inv R_2\D(E_1)=&\D'(E_1)W_0 w_2\inv R_2w_1\inv R_1w_2\inv R_2,
}
or simplifying:
\Eq{w_1\inv R_1w_2\inv R_2\D(E_1)=&\D(q^{\half}E_2K_2^{\half})w_1\inv R_1w_2\inv R_2,\\
w_1\inv R_1w_2\inv R_2\D(F_1)=&\D(q^{\half}F_2K_2^{-\half}) w_1\inv R_1w_2\inv R_2,
}
and also
\Eq{w_1\inv R_1\D(E_1)=\D'(F_1)w_1\inv R_1.}
Apply this repeatedly, we can show the braiding relation for all other simply-laced type.

\subsection{Quasi-triangularity relations in simply-laced case}\label{sec:R:qt}
Again let's work with $U_{q\til{q}}(\sl(3,\R))$. We will prove the first relation $$(\D\ox 1)R=R_{13}R_{23},$$ the second one is similar. We have
{
\allowdisplaybreaks
\Eqn{
(\D\ox 1)R=&\D(\bQ^\half)(g_b(\D\be_2\ox \bf_2)g_b(\D\be_{12}\ox\bf_{12})g_b(\D\be_1\ox\bf_1)\D(\bQ^\half)\\
=&\bQ_{13}^\half \bQ_{23}^\half g_b(\be_2\ox K_2^{\half}\ox \bf_2+K_2^{-\half}\ox \be_2\ox \bf_2)\\
&\cdot g_b(\be_{12}\ox K_1^{\half}K_2^{\half}\ox \bf_{12}+K_2^{-\half}\be_1\ox K_1^{\half}\be_2\ox \bf_{12}+K_1^{-\half}K_2^{-\half}\ox \be_{12}\ox \bf_{12})\\
&\cdot g_b(\be_1\ox K_1^{\half}\ox \bf_1+K_1^{-\half}\ox \be_1\ox \bf_1)\bQ_{13}^\half \bQ_{23}^\half \\
=&\bQ_{13}^\half \bQ_{23}^\half g_b(\be_2\ox K_2^{\half}\ox \bf_2)g_b(K_2^{-\half}\ox \be_2\ox \bf_2)\\
&\cdot g_b(\be_{12}\ox K_1^{\half}K_2^{\half}\ox \bf_{12})g_b(K_2^{-\half}\be_1\ox K_1^{\half}\be_2\ox \bf_{12})g_b(K_1^{-\half}K_2^{-\half}\ox \be_{12}\ox \bf_{12})\\
&\cdot g_b(\be_1\ox K_1^{\half}\ox \bf_1)g_b(K_1^{-\half}\ox \be_1\ox \bf_1)\bQ_{13}^\half \bQ_{23}^\half \\
=&\bQ_{13}^\half \bQ_{23}^\half g_b(\be_2\ox K_2^{\half}\ox \bf_2)g_b(\be_{12}\ox K_1^{\half}K_2^{\half}\ox \bf_{12})\\
&\cdot g_b(K_2^{-\half}\ox \be_2\ox \bf_2)g_b(K_2^{-\half}\be_1\ox K_1^{\half}\be_2\ox \bf_{12})g_b(\be_1\ox K_1^{\half}\ox \bf_1)\\
&\cdot g_b(K_1^{-\half}K_2^{-\half}\ox \be_{12}\ox \bf_{12})g_b(K_1^{-\half}\ox \be_1\ox \bf_1)\bQ_{13}^\half \bQ_{23}^\half \\
=&\bQ_{13}^\half \bQ_{23}^\half g_b(\be_2\ox K_2^{\half}\ox \bf_2)g_b(\be_{12}\ox K_1^{\half}K_2^{\half}\ox \bf_{12})\\
&\cdot g_b(\be_1\ox K_1^{\half}\ox \bf_1)g_b(K_2^{-\half}\ox \be_2\ox \bf_2)\\
&\cdot g_b(K_1^{-\half}K_2^{-\half}\ox \be_{12}\ox \bf_{12})g_b(K_1^{-\half}\ox \be_1\ox \bf_1)\bQ_{13}^\half \bQ_{23}^\half \\
=&\bQ_{13}^\half g_b(\be_2\ox 1\ox \bf_2)g_b(\be_{12}\ox 1\ox \bf_{12})g_b(\be_1\ox 1\ox \bf_1)\bQ_{13}^\half \\
&\cdot \bQ_{23}^\half g_b(1\ox \be_2\ox \bf_2)g_b(1\ox \be_{12}\ox \bf_{12})g_b(1\ox \be_1\ox \bf_1)\bQ_{23}^\half \\
=&R_{13}R_{23}.
}
}
Where in the fourth line we used
$$\frac{[K_2^{-\half}\ox \be_2\ox \bf_2, \be_1\ox K_1^{\half}\ox \bf_1]}{q-q\inv}=K_2^{-\half}\be_1\ox K_1^{\half}\be_2\ox \bf_{12}.$$

By the relations $R_{12}=w_1 R_2 w_1\inv$, the above calculation is also equivalent to the following relations of the quantum dilogarithms:
\Eqn{
&g_b(K_2^{-\half}\be_1\ox K_1^{\half}\be_2\ox \bf_{12})\\
=& g_b^*(K_2^{-\half}\ox \be_2\ox \bf_2)g_b(\be_1\ox K_1^{\half}\ox \bf_1)g_b(K_2^{-\half}\ox \be_2\ox \bf_2)g_b^*(\be_1\ox K_1^{\half}\ox \bf_1)\\
=&g_b^*(\be_{12}\ox K_1^{\half}K_2^{\half}\ox \bf_{12})g_b(K_2^{\half}\bf_2\ox K_2^{-\half}\be_2\ox 1)g_b(\be_{12}\ox K_1^{\half}K_2^{\half}\ox \bf_{12})g_b^*(K_2^{\half}\bf_2\ox K_2^{-\half}\be_2\ox 1),
}
or after rewriting,
$ g_b(K_2^{-\half}\ox \be_2\ox \bf_2)g_b(K_2^{\half}\bf_2\ox K_2^{-\half}\be_2\ox 1)$ commute with \\$g_b(\be_{12}\ox K_1^{\half}K_2^{\half}\ox \bf_{12})g_b(\be_1\ox K_1^{\half}\ox \bf_1)$.
Symbolically, we present this relation informally as \Eq{G_{23}^2G_{21}^{'2}G_{13}^{12}G_{13}^1=G_{13}^{12}G_{13}^1G_{23}^2G_{21}^{'2},}
which resembles the so-called ``Tetrahedron Equation" \cite{KV}. It suffices to apply this relation, together with \eqref{wR} repeatedly to obtain the quasi-triangular relation in higher rank.

\subsection{Remarks on the non-simply-laced case}\label{sec:R:BCFG}
The relations for the non-simply-laced case can also be done along the same line. What we have found is that the braiding relations amount to generalized pentagon relation for $g_b$ given by Proposition \ref{UVcd}, and the same relations apply to all higher rank case. 

On the other hand, the quasi-triangularity is more difficult. For type $B_2$, it is equivalent to the generalized exponential relation given in Proposition \ref{genexp}, which is needed to break down the coproduct of $\be_{21}$ and $\be_{12}$. For simplicity, let
\Eq{\be_3':=\be_{121}=\be_{2\inv 1}=\frac{q^\half \be_2\be_1-q^{-\half} \be_1\be_2}{q-q\inv},}
\Eq{\be_X:=\be_{12}=\frac{\be_3'\be_1-\be_1\be_3'}{q^\half-q^{-\half}}.}
(Recall $q=e^{\pi \bi b_l^2}=q_2$ and $q^\half=e^{\pi \bi b_s^2}=q_1$.)
Then $R$ is given by
\Eq{R=\bQ^\half g_b(\be_2\ox \bf_2)g_b(\be_3'\ox \bf_3')g_b(\be_X\ox \bf_X)g_b(\be_1\ox \bf_1)\bQ^\half.}
Proposition \ref{genexp} then implies
\Eqn{g_{b_s}(\D(\be_3')\ox \bf_3')=&g_{b_s}(\be_3'\ox K_3^\half\ox \bf_3')g_{b_l}(\be_X K_2^{-\half}\ox \be_2K_X^{\half}\ox {\bf_3'}^{2})\\
&\tab \cdot g_{b_s}(\be_1K_2^{-\half}\ox \be_2K_1^\half\ox \bf_3')g_{b_s}(K_3^{-\half}\ox \be_3'\ox \bf_3')}
\Eqn{g_{b_l}(\D(\be_X)\ox \bf_X)=&g_{b_l}(\be_X\ox K_X^{\half}\ox \bf_X)g_{b_l}(\be_1^2K_2^{-\half}\ox \be_2K_1\ox \bf_X),\\
&\tab \cdot g_{b_s}(\be_1K_3^{-\half}\ox \be_3'K_1^\half\ox \bf_X)g_{b_l}(K_X^{-\half}\ox \be_X\ox \bf_X),}
and together with the generalized pentagon relation the quasi-triangularity can be proved. Again these can be rephrased as a generalized Tetrahedron equation using the quantum Weyl element, and we believe these are all we need to prove the higher rank case. 

\section{$\cU_{q\til{q}}(\g_\R)$ as a quasi-triangular multiplier Hopf algebra}\label{sec:Cstar}
So far we have worked on the algebraic calculation quite formally. From the explicit expression of the $R$ operator in Theorem \ref{mainR}, it motivates us to define $R$ as the canonical element of certain Drinfeld's double construction. The accurate language to use here turns out to be the so-called multiplier Hopf algebra \cite{VD} and its Drinfeld's double construction \cite{DvD}, which gives the notion of a quasi-triangular multiplier Hopf algebra introduced by \cite{Zh}.

Let us recall the basic definitions. For further details please refer to \cite{VD}.
\begin{Def}
Let $\cB(\cH)$ be the algebra of bounded linear operators on a Hilbert space $\cH$. Then the multiplier algebra $M(\cA)$ of a $C^*$-algebra $\cA\subset \cB(\cH)$ is the $C^*$-algebra of operators
\Eq{M(\cA)=\{b\in \cB(\cH): b\cA\subset \cA, \cA b\subset \cA\}.}
In particular, $\cA$ is an ideal of $M(\cA)$.
\end{Def}
\begin{Def}\label{mHa}
A multiplier Hopf *-algebra is a $C^*$-algebra $\cA$ together with the antipode $S$, the counit $\e$, and the coproduct map
\Eq{\D:\cA\to M(\cA\ox \cA),}
all of which can be extended to a map from $M(\cA)$, such that the usual properties of a Hopf algebra holds on the level of $M(\cA)$.
\end{Def}
\begin{Def}\label{qtmHa}
A quasi-triangular multiplier Hopf algebra is a multiplier Hopf algebra $\cA$ together with an invertible element $R\in M(\cA\ox \cA)$ such that 
\Eq{(\D \ox id)(R)=&R_{13}R_{23}\in M(\cA\ox \cA\ox \cA),\\
(id\ox\D)(R)=&R_{13}R_{12} \in M(\cA\ox \cA\ox \cA),\\
\D'(a)R=&R\D(a)\in M(\cA\ox \cA),\tab \forall a\in M(\cA),\\
(\e\ox id)(R)=&(id\ox \e)(R)=1\in M(\cA).
}
\end{Def}
Furthermore, the element $u:=m^{op}(1\ox S)(R)$ will be an invertible element in $M(\cA)$ such that 
\Eq{S^2(a)=uau\inv,\tab \forall a\in M(\cA).} 
\begin{Def}\label{rmHa}
A ribbon multiplier Hopf algebra is a quasi-triangular multiplier Hopf algebra $\cA$ that possesses a central ribbon element $v\in M(\cA)$, such that
\Eq{v^2=&uS(u), \tab S(v)=v,\tab \e(v)=1,\\
\D(v)=&(R_{21}R_{12})\inv (v\ox v)
}
hold in $M(\cA)$.
\end{Def}

\subsection{The Borel subalgebra $\cU_{q\til{q}}^{C^*}(\fb_\R)$}\label{sec:Cstar:borel}
Let us fix a positive representation $\cP_\l$ of $\cU_{q\til{q}}(\g_\R)$, and fix a reduced expression $w_0=s_{i_1}...s_{i_N}$ of the longest element. Motivated from the compact case, as well as the expression of $R$, it is intuitive to choose a ``basis" given by
\Eq{\prod_i H_i^{m_i}\prod_{k=1}^N \be_{\a_k}^{\bi b_{i_k}\inv t}=H_1^{m_1}...H_n^{m_n}\be_{\a_N}^{\bi b_{i_N}\inv t}...\be_{\a_1}^{\bi b_{i_1}\inv t}.}
Here $N=l(w_0)$, while $n=rank(\g)$, and as before
\Eq{\be_{\a_k}:=T_{i_1} T_{i_2}... T_{i_{k-1}} \be_{i_k}.}

Following the approach in \cite{Ip1} for the harmonic analysis of the quantum plane, we give the following definition.
\begin{Def} \label{Ub}We define the $C^*$-algebraic version of the Borel subalgebra $$\bU\bb:=\cU_{q\til{q}}^{C^*}(\fb_\R^+)$$ as the operator norm closure of the linear span of all bounded operators on $L^2(\R^N)$ of the form
\Eq{\vec[F]:= F_0(\bH)\prod_{k=1}^N  \int_C \frac{F_k(t_k)}{G_{b_{i_k}}(Q_{i_k}+\bi t_k)} \over[\be]_{\a_k}^{\bi b_{i_k}\inv t_k}dt_k,}
where 
\Eq{\be_{\a_k}=T_{i_1} T_{i_2}... T_{i_{k-1}} \be_{i_k},\tab \over[\be]_{\a_k}=q_{i_k}^{-\half} K_{\a_k}^\half\be_{\a_k}}
\Eq{F_0(\bH):=F_0(\bi b_1H_1,...,\bi b_nH_b)}
is a smooth compactly supported functions on the positive operators $\bi b_k H_k$,
$F_k(t_k)$ are entire analytic functions that have rapid decay along the real direction (i.e. for fixed $y_0$, $F_k(x+\bi y_0)$ decays faster than any exponential function in $x$.). Finally the contour $C$ is along the real axis which goes above the pole at $t_k=0$.
\end{Def}
Since $\over[\be]_{\a_k}$ are positive essentially self-adjoint, $\over[\be]_{\a_k}^{\bi b_{i_k}\inv t}$ is unitary, and by the decay properties of $F_k$, the operator $\vec[F]$ is indeed a well-defined bounded operator acting on $L^2(\R^N)$. Furthermore, since the positive representations are injective, the definition of this algebra does not depend on the choice of the parameter $\l$. Finally by Proposition \ref{interchange} below, the usual complex conjugation gives the star structure of $\bU\bb$.

\begin{Rem} Definition \ref{Ub} is compatible with the modular double counterpart. In other words, we obtain the same algebra when we replace all variables $\be_{\a_k}, \bi b_iH_i$ with $\til{\be}_{\a_k}, \bi b_i\inv H_i$ due to the transcendental relations. Hence $\bU\bb$ can indeed be called the modular double of the Borel subalgebra.
\end{Rem}

\begin{Prop} The map defined by
\Eq{\D(\vec[F])=F_0(\D(\bH))\prod_{k=1}^N  \int_C \frac{F_k(t_k)}{G_{b_{i_k}}(Q_{i_k}+\bi t_k)} \D(\over[\be]_{\a_k})^{\bi b_{i_k}\inv t_k}dt_k}
is a coproduct $\D:\cA\to M(\cA\ox \cA)$, where $\D(H_i)=H_i\ox 1+1\ox H_i$.
\end{Prop}
\begin{proof}
Coassociativity is immediate since the expression is the same as in the usual case. The factors $G_{b_{i_k}}(Q_{i_k}+\bi t_k)$ are needed in order to define the coproduct as a multiplier Hopf algebra. This follows from the use of the $q$-binomial formula (Lemma \ref{qbinom}), or in the non-simply-laced case, the generalized exponential relation (Proposition \ref{genexp}), which basically says that $\D(\over[\be]_{\a_k}^{\bi b_{i_k}\inv t_k})=\D(\over[\be]_{\a_k})^{\bi b_{i_k}\inv t_k}$ cancels the factors $G_{b_i}(Q_{i_k}+\bi t_k)$ and introduce two new factors $G_{b_i}(Q_{i_k}+i\t_1)G_{b_i}(Q_{i_k}+i\t_2)$ in the respective factors for $\over[\be]_{\a_k}^{\bi b_{i_k}\inv \t_1}\ox \over[\be]_{\a_k}^{\bi b_{i_k}\inv \t_2}$. For non-simple roots, the extra integration can be shown to be holomorphic due to meromorphicity of $G_b$ as well as application of the delta distribution rules (Proposition \ref{delta}).
\end{proof}

For the term $\prod_{k=1}^N \over[\be]_{\a_k}^{\bi b_{i_k}\inv t}$ to deserve to be called a ``basis", it suffices to show that we can interchange the order of the generators. Only the rank=2 cases need to be considered, we show this for the non-bar generators in the simply-laced case as follows.
\begin{Prop}\label{interchange} In type $A_2$, we have
\Eq{\label{121a}\frac{\be_2^{\bi b\inv t}\be_1^{\bi b\inv s}}{G_b(Q+\bi s)G_b(Q+\bi t)}=\int_{C}\frac{e^{\pi \bi (2s\t+t\t-st-\frac{5}{2}\t^2)}\be_1^{\bi b\inv(s-\t)}\be_{21}^{\bi b\inv \t}\be_2^{\bi b\inv (t-\t)}}{G_b(Q+\bi s-\bi \t)G_b(Q+\bi \t)G_b(Q+\bi t-\bi \t)}d\t,}
where the contour separate the poles of $\t=0$ and $\t=s,t$.
\Eq{\label{121b}\frac{\be_{12}^{\bi b\inv t}}{G_b(Q+\bi t)}=\int_{C}\frac{e^{\pi Q(\t-t)+\pi \bi \t t-\frac{3}{2}\pi \bi \t^2}\be_1^{\bi b\inv \t}\be_{21}^{\bi b\inv(t-\t)}\be_2^{\bi b\inv \t}}{G_b(Q+\bi t-\bi \t)G_b(Q+\bi \t)} d\t,}
where the contour separate the poles of $\t=0$ and $\t=t$.
\end{Prop}
Note that by taking $s,t\to -\bi b$, one recovers the standard relation
\Eq{\be_2\be_1=&q\be_1\be_2-(q-q\inv)q^{\frac{1}{2}}\be_{21},\\
\be_{12}=&q^{\frac{1}{2}}\be_1\be_2-q\be_{21},
}
 by means of Proposition \ref{delta}. Also, the factors $G_b(Q+\bi t)$ implies that the holomorphicity condition for $\bU\bb$ is still satisfied.
\begin{proof}
By the generalized pentagon relation \eqref{genpentagon}, we have
$$g_b(q^{-\half} K_2^\half \be_2)g_b(q^\half K_1^{-\half}\be_1)=g_b(q^\half K_1^{-\half}\be_1)g_b(K_1^{-\half}K_2^{\half}\be_{21})g_b(q^{-\half}K_2^\half \be_2).$$
Now expand the relation using Lemma \ref{FT}, and equate the powers of $K_1$ and $K_2$ we will obtain \eqref{121a}.

Next, using again the generalized pentagon relation again, written as
 $$g_b^*(q^\half K_2^{-\half}\be_2)g_b(q^{-\half} K_1^\half \be_1)g_b(q^\half K_2^{-\half}\be_2)g_b^*(q^{-\half}K_1^\half \be_1)=g_b(K_2^{-\half}K_1^{\half}\be_{12}),$$
expanding by Lemma \ref{FT} and equating again the powers of $K_1$ and $K_2$, and using the first equation to interchange $\be_1$ and $\be_2$, the integral can be evaluated explicitly and we obtain \eqref{121b}. 
\end{proof}

The interchange relation in the non-simply-laced case can be obtained along the same line by combining Proposition \ref{UVcd} and \ref{genexp}.

As a collorary, we can now define the antipode.
\begin{Def} The antipode is defined on the generators by (cf. \eqref{SS})
\Eq{S(H_i)=&-H_i\\
S(\be_i^{\bi b_i\inv t})=&e^{-\pi Q_it} \be_i}
and extended anti-homomorphically.
\end{Def}
So for example we have
$S(\be_{12}^{\bi b\inv t})=e^{-2\pi Qt}\be_{21}^{\bi b\inv t}$.
\begin{Cor} By the interchange relation from Proposition \ref{interchange}, the antipode is a map from $\bU\bb$ to $\bU\bb$. Furthermore, together with the complex conjugation properties \eqref{Gbcomplex} of $G_b(x)$, $\bU\bb$ also possesses a natural star structure given by complex conjugation, such that the analytic properties are satisfied.
\end{Cor}
Finally, we define
\Eq{\e(\vec[F])=F_0(\zero),}
by setting all $H_i$ to be zero.
\begin{Cor} The $C^*$-algebra $\bU\bb$ is a multiplier Hopf *-algebra in the sense of Definition \ref{mHa}.
\end{Cor}
\subsection{Hopf pairing and Drinfeld's double}\label{sec:Cstar:pairing}
For two Hopf algebra $\cA,\cA'$, a pairing is called a Hopf pairing if for $a\in \cA,b,c\in \cA'$,
\Eq{\<a,bc\>=\<\D(a),b\ox c\>=\sum \<a^i, b\>\<a_i,c\>,}
\Eq{\<S(a),b\>=\<a,S(b)\>,}
\Eq{\<a,1\>=\e(a), \tab\<1,b\>=\e(b),}
where $\D(a)=\sum a^i\ox a_i$. Moreover, it can be extended naturally to the multiplier algebra $M(\cA)$.
Let $\bU\bb^-$ be the multiplier Hopf algebra generated in the above sense by $\bi b_iW_i$ and $\bf_{\a_k}$ with the opposite coproduct. Then we define the pairing on the generators (we used the modified generator, cf. Remark \ref{bar}) as
\begin{Def} The Hopf pairing is given only on generators with same indices by
\Eq{\left\<(\bi b_iH_i)^n, (\bi b_iW_i)^m\right\>=&\d_{mn}n!\frac{\bi }{\pi},\\
\left\<\over[\be]_{\a_k}^{\bi b_{i_k}\inv s}, \over[\bf]_{\a_k}^{\bi b_{i_k}\inv t}\right\>=&\d(s-t)G_{b_{i_k}}(Q_{i_k}+\bi t)e^{\pi \bi t^2},
}
or more generally, denoting $\vec[F]\in \bU\bb, \vec[F']\in \bU\bb^-$, 
\Eq{\<\vec[F],\vec[F']\>=\<F_0(\bH),F_0'(\bW)\>\prod_{k=1}^N\int \frac{F_{i_k}(t_k)F_{i_k}'(t_k)e^{\pi \bi t_k^2}}{G_{b_{i_k}}(Q_{i_k}+\bi t_k)} dt_k.}
\end{Def}

We will show that the definition is consistent with the Hopf pairing between simple root generators. Using the $q$-binomial theorem (Lemma \ref{qbinom}), we have:
\Eqn{&\<\over[\be]^{\bi b\inv s}, \over[\bf]^{\bi b\inv t}\>=\<\over[\be]^{\bi b\inv s}, \over[\bf]^{\bi b\inv t_1}\over[\bf]^{\bi b\inv t_2}\>\\
&=\<\D(\over[\be]^{\bi b\inv s}), \over[\bf]^{\bi b\inv t_1}\ox\over[\bf]^{\bi b\inv t_2}\>\\
&=\left<\int_C \frac{G_b(-\bi s+\bi \t)G_b(-\bi \t)}{G_b(-\bi s)} \over[\be]^{\bi b\inv \t}\ox K^{\bi b\inv \t}\over[\be]^{\bi b\inv(s-\t)},\over[\bf]^{\bi b\inv t_1}\ox\over[\bf]^{\bi b\inv t_2}\right>\\
&= \int_C\frac{G_b(-\bi s+\bi \t)G_b(-\bi \t)}{G_b(-\bi s)}\d(\t-t_1)e^{\pi \bi t_1^2}G_b(Q+\bi t_1)\d(s-\t-t_2)e^{\pi \bi t_2^2}G_b(Q+\bi t_2)d\t\\
&=\frac{G_b(-\bi t_2)G_b(-\bi t_1)G_b(Q+\bi t_1)G_b(Q+\bi t_2)}{G_b(-\bi s)}e^{\pi \bi t_1^2+\pi \bi t_2^2}\d(s-t_1-t_2)\\
&=e^{\pi \bi s^2}G_b(Q+\bi s)\d(s-(t_1+t_2)).}
The other case is similar. The properties involving antipode is easy to check if we choose the reverse ordering of the basis on $\bU\bb^-$. The properties of $\e$ is trivial.

Now we recall the Drinfeld's double construction in the setting of multiplier Hopf algebra.
\begin{Def} \cite{DvD} The Drinfeld's double $\cD$ of multiplier Hopf algebra $\cA$ and its dual $\cA'$ is a Hopf algebra with underlying vector space $\cA\ox \cA'$, comultiplication $\D_\cA\ox \D_{\cA'}^{op}$, and product given by
\Eq{(a\ox x)(b\ox y)=\sum ab_{(2)}\ox x_{(2)}y\<b_{(1)},S_{\cA'}\inv (x_{(3)})\>\<b_{(3)},x_{(1)}\>.}
\end{Def}
Then it is known \cite{DvD2} that the the Drinfeld's double $\cD$ is a quasi-triangular multiplier Hopf algebra, where $R$ is given by the canonical element, which is the unique element in $M(\cA\ox\cA')\subset M(\cD\ox \cD)$ such that 
\Eq{\<R,b\ox a\>=\<a,b\>,\tab a\in\cA, b\in \cA'.}
\begin{Def} We define 
\Eq{\bU:=\cU_{q\til{q}}^{C^*}(\g_\R):=\cD(\bU\bb)/(W_i=(A\inv)_{ij}H_j)}
 to be the Drinfeld's double of the Borel subalgebra $\bU\bb$ modulo the Cartan subalgebra $\fh\subset \bU\bb^-$.
\end{Def}
\begin{Cor}\label{mainRCor} $\bU$ is a quasi-triangular multiplier Hopf algebra. The canonical element is given precisely by $R$ as in Theorem \ref{mainR}. This follows directly from the explicit expression of $R$ and the Hopf pairing we are using.
\end{Cor}
Finally we note that $R$ acts as unitary operator on the positive representations $P_{\l_1}\ox P_{\l_2}$ giving the braiding structure. 

\subsection{The ribbon structure of $\what{\cU}_{q\til{q}}(\g_\R)$}\label{sec:Cstar:ribbon}
In Section \ref{sec:QWeyl:ribbon}, we have computed in the case of $\cU_{q\til{q}}(\sl(2,\R))$ the element \\$u=m^{op}(1\ox S)(R)$ to be
$$u=vK^{\frac{Q}{b}},$$
which is now clear that it lies in the multiplier algebra $M(\cU_{q\til{q}}^{C^*}(\sl(2,\R)))$ in the sense defined in the previous subsection. Let us adjoin the unitary operators $w_1,...,w_n$ defined in \eqref{ww} to the algebra $\bU$, and call this $\what{\cU}_{q\til{q}}(\g_\R)$.
\begin{Prop}
Define $v=W_0^2$ where $W_0=w_{i_1}...w_{i_N}$ with $s_{i_1}...s_{i_N}$ a reduced expression of the longest element. Then $v$ is a ribbon element, making $\what{\cU}_{q\til{q}}(\g_\R)$ a ribbon multiplier Hopf algebra.
\end{Prop}
The properties of $v$ follows directly from the coproduct properties of $w_i$, and the fact that $W_0^2$ commute with all the generators $\be_i, \bf_i, K$, hence $v$ is central. Furthermore, the operator $u$ can now be expressed as
\Eq{u=v\prod_{i=1}^n K_i^{\frac{Q_i}{b_i}},}
which is again clear that it lies in the multiplier $M(\bU)$. 

With the involvement of $Q^2$ in the expression of $v$, this means that there are no classical limit as $b\to 0$, and we believe that this observation opens up a possibility of finding a new class of quantum topological invariants, where the ribbon structure plays a crucial role \cite{RT1, RT2}.


\end{document}